\documentclass[12pt,oneside,english]{amsart}

\usepackage[T1]{fontenc}
\usepackage[utf8]{inputenc}
\usepackage{lmodern}
\usepackage{babel}
\usepackage{amsmath,amssymb,amsthm,mathtools}
\usepackage{booktabs}
\usepackage{array}
\usepackage{graphicx}
\usepackage{geometry}
\usepackage{tikz}
\usepackage{tikz-3dplot}
\usetikzlibrary{positioning, calc, arrows.meta, decorations.markings, intersections, backgrounds, 3d, shapes.geometric, fit}
\usepackage[colorlinks=true,linkcolor=blue,citecolor=blue,urlcolor=blue]{hyperref}

\graphicspath{{Images/}}

\numberwithin{equation}{section}
\numberwithin{figure}{section}
\numberwithin{table}{section}

\theoremstyle{plain}
\newtheorem{theorem}{Theorem}[section]
\newtheorem{proposition}[theorem]{Proposition}
\newtheorem{lemma}[theorem]{Lemma}
\theoremstyle{definition}
\newtheorem{definition}[theorem]{Definition}
\theoremstyle{remark}
\newtheorem{remark}[theorem]{Remark}

\newcommand{\ZZ}{\mathbb Z}
\newcommand{\QQ}{\mathbb Q}
\newcommand{\RR}{\mathbb R}
\newcommand{\CC}{\mathbb C}
\newcommand{\HH}{\mathbb H}
\newcommand{\Oct}{\mathbb O}
\newcommand{\KK}{\mathbb Q(\sqrt5)}
\newcommand{\Zphi}{\mathbb Z[\varphi]}
\newcommand{\Tr}{\operatorname{Tr}}
\newcommand{\Norm}{\operatorname{N}}

\newcommand{\Span}{\operatorname{span}}

\newcommand{\Rea}{\operatorname{Re}}

\begin{document}

\title[Non-crystallographic systems of integers]{Non-crystallographic systems of integers over composition algebras}
\author{Daniele Corradetti}
\address{Grupo de F\'\i sica Matem\'atica, Instituto Superior T\'ecnico, Av. Rovisco Pais, 1049-001 Lisboa, Portugal}
\email{danielecorradetti@tecnico.ulisboa.pt}
\address{Departamento de Matematica, Universidade do Algarve, Campus de Gambelas, 8005-139 Faro, Portugal}
\email{a55499@ualg.pt}
\subjclass[2020]{11R04; 17A35; 17A75; 20F55; 52C23}
\keywords{composition algebras, root systems, golden integers, icosians, quasicrystals}

\begin{abstract}
In this work we revisit classical systems of integers inside the real normed division algebras from the point of view of finite norm shells and root systems. Building on the icosian framework of Moody--Patera and on the integral root-system viewpoint of Chen--Moody--Patera and of Johnson, we isolate the precise axiomatic ingredients of the non-crystallographic analogue: an order over the golden ring \(\Zphi\) together with a distinguished finite root shell whose Cartan coefficients lie in \(\Zphi\). We show that the usual Gaussian, Eisenstein, Hamilton, Hurwitz and Coxeter--Dickson examples are recovered by separating the order, its units, and its distinguished finite shells; once the lattice requirement is replaced by a finite root-shell requirement, the golden integer ring becomes the natural coefficient ring for the non-crystallographic cases \(H_2\) and \(H_4\). We then construct a weak golden octonion order by Cayley--Dickson doubling of the icosian ring; the resulting free rank-\(8\) \(\Zphi\)-order has a \(240\)-element finite shell of type \(H_4\oplus H_4\) and its multiplication is genuinely octonionic. Finally, we prove (i) that this weak double is self-dual with respect to the polar norm pairing, hence has no strict norm-integral overorder, and (ii) that the first trace-integral discriminant tower over it contains no octonion-stable nonzero isotropic gluing.
\end{abstract}

\maketitle
\tableofcontents

\section{Introduction and Motivations}

Unital composition algebras, and in particular the four real normed division algebras \(\RR\), \(\CC\), \(\HH\) and \(\Oct\), are among the most fascinating places where algebra, geometry and symmetry meet. They give rise to normed products, involutions, exceptional lattices, reflection groups and, in the octonionic case, to a rich non-associative geometry. The existence of integral elements in these algebras is therefore not merely an arithmetic question, but also a way of spotting finite geometric configurations inside continuous algebras.

The classical story starts from the usual integers \cite{Gau} and continues with Gaussian integers, Eisenstein integers, Hamilton quaternions \cite{Hu19}, Hurwitz quaternions and integral Cayley numbers. Dickson's formulation \cite{Di23} asks for subsets closed under addition and multiplication, containing \(1\), and enjoying integral trace and norm. Coxeter's classification of reflection groups \cite{Coxeter1934,Co46} and Johnson's more recent reorganization \cite{Jo13,Jo17,Jo18} then emphasized a more rigid feature: the finite unit shells of the most symmetric examples recover crystallographic lattices and root systems. The geometry of finite reflection groups behind this picture is standard \cite{Humphreys}, and the dictionary between quaternionic and octonionic orders, on the one hand, and discrete reflection groups, on the other, has been extensively developed for octonions \cite{CoSm}.

The following synoptic table summarizes the crystallographic landscape that will be reviewed in this work. It is important to read the column ``shell'' as the distinguished finite norm shell and not as the whole order.

\begin{table}[h]
\centering
\resizebox{\textwidth}{!}{%
\begin{tabular}{llllll}
\toprule
Name & Algebra & Order symbol & Shell size & Multiplicative object & Root type \\
\midrule
Integers & \(\RR\) & \(\ZZ\) & \(2\) & \(C_2\) & \(A_1\) \\
Gaussian & \(\CC\) & \(\ZZ[i]\) & \(4\) & \(C_4\) & \(A_1\oplus A_1\) \\
Eisenstein & \(\CC\) & \(\ZZ[\omega]\) & \(6\) & \(C_6\) & \(A_2\) \\
Hamilton & \(\HH\) & \(\ZZ[i,j,k]\) & \(8\) & \(Q_8\) & \(A_1^{\oplus4}\) \\
Hybrid & \(\HH\) & \(\ZZ[\omega]\oplus \ZZ[\omega]j\) & \(12\) & \(\operatorname{Dic}_3\) & \(A_2\oplus A_2\) \\
Hurwitz & \(\HH\) & \(\mathcal H\) & \(24\) & \(2T\) & \(D_4\) \\
Graves--Cayley & \(\Oct\) & \(\mathcal C\) & \(16\) & Moufang loop & \(A_1^{\oplus8}\) \\
Eisenstein octaves & \(\Oct\) & \(\mathcal E\) & \(24\) & Moufang loop & \(4A_2\) \\
Coupled Hurwitz & \(\Oct\) & \(\mathcal K\) & \(48\) & Moufang loop & \(2D_4\) \\
Coxeter--Dickson & \(\Oct\) & \(\mathcal O_{E_8}\) & \(240\) & Moufang loop & \(E_8\) \\
\bottomrule
\end{tabular}
}
\caption{\emph{This table summarizes the classical crystallographic systems of integer elements over real normed division algebras. The non-associative rows must be interpreted as Moufang-loop shells, not as groups.}}
\label{tab:crystallographic-summary}
\end{table}

While the octonionic case leads to the exceptional root system \(E_8\), the icosian case leads to the non-crystallographic root system \(H_4\). This analogy is the guiding principle of the present work. In the crystallographic case the coefficient ring is usually \(\ZZ\), and the roots generate a genuine lattice. In the non-crystallographic case the Cartan coefficients are no longer integers, but algebraic integers in the golden ring \(\Zphi\). Thus the relevant object is no longer a lattice over \(\ZZ\), but a finite root shell over \(\Zphi\).

The golden side of the story can be summarized as follows.

\begin{table}[h]
\centering
\resizebox{\textwidth}{!}{%
\begin{tabular}{llll}
\toprule
Object & Arithmetic support & Distinguished shell & Role in this work \\
\midrule
\(H_2\) & \(\ZZ[\zeta_{10}]\), real part \(\Zphi\) & \(10\) roots & decagonal model and Penrose arithmetic \\
\(H_3\) & \(\Zphi\)-coordinate module & \(30\) roots & icosahedral star geometry \\
Icosian ring & \(\mathbb I\subset\HH(K)\) & \(120\) roots & quaternionic model of \(H_4\) \\
Icosian double & \(\mathbb G=\mathbb I+\mathbb I\ell\) & \(H_4\oplus H_4\) & weak golden octonion order \\
Strict G3 over \(G_0\) & norm-polar \(\Zphi\)-dual & no new overorder & closed by \(G_0^\#=G_0\) \\
Trace G3-B' tower & trace-polar \(\ZZ\)-dual & no stable \(\sqrt5\)-gluing & bounded no-go inside the \(G_0\) tower \\
\bottomrule
\end{tabular}
}
\caption{\emph{This table summarizes the non-crystallographic arithmetic landscape considered in the present work.}}
\label{tab:noncrystallographic-summary}
\end{table}

The two tables above are summarized graphically in Fig.~\ref{fig:nc-tower}: each Cayley--Dickson step climbs one row, while passing from the left column to the right column amounts to replacing the coefficient ring $\ZZ$ with the golden ring $\Zphi$.

\begin{figure}[htbp]
\centering
\begin{tikzpicture}[
  every node/.style={font=\small},
  alg/.style={font=\bfseries\small, anchor=center},
  box/.style={
    rectangle, rounded corners=3pt, draw=black, thick,
    fill=white, align=center, minimum width=3.2cm, minimum height=0.95cm,
    inner sep=3pt,
  },
  cdarrow/.style={-{Latex[length=2mm]}, black, thick},
  hr/.style={dashed, black},
]
\node[font=\small\bfseries, black] at (0,3.7) {Crystallographic};
\node[font=\small\bfseries, black] at (5.0,3.7) {Non-crystallographic};
\node[font=\small\itshape, black] at (0,3.2) {coefficient ring $\ZZ$};
\node[font=\small\itshape, black] at (5.0,3.2) {coefficient ring $\Zphi$};

\node[alg, anchor=east] at (-2.5,2.4) {$\RR,\,\CC$:};
\node[alg, anchor=east] at (-2.5,1.0) {$\HH$:};
\node[alg, anchor=east] at (-2.5,-0.4) {$\Oct$:};

\node[box] (c1) at (0,2.4) {$\ZZ[i],\;\ZZ[\omega]$\\\footnotesize $|S_1|=4,\,6$};
\node[box] (c2) at (0,1.0) {Hurwitz $\mathcal H$\\\footnotesize $|S_1|=24$, type $D_4$};
\node[box] (c3) at (0,-0.4) {Coxeter--Dickson $\mathcal O_{E_8}$\\\footnotesize $|S_1|=240$, type $E_8$};

\node[box] (n1) at (5.0,2.4) {cyclotomic $H_2$ in $\ZZ[\zeta_{10}]$\\\footnotesize $|S_{H_2}|=10$};
\node[box] (n2) at (5.0,1.0) {icosian ring $\mathbb I$\\\footnotesize $|S_{H_4}|=120$, type $H_4$};
\node[box] (n3) at (5.0,-0.4) {icosian double $\mathbb G=\mathbb I+\mathbb I\ell$\\\footnotesize $|S_{\mathbb G}|=240$, type $H_4\oplus H_4$};

\draw[cdarrow] (c1) -- node[right, font=\footnotesize, black]{C-D} (c2);
\draw[cdarrow] (c2) -- node[right, font=\footnotesize, black]{C-D} (c3);
\draw[cdarrow] (n1) -- node[right, font=\footnotesize, black]{C-D} (n2);
\draw[cdarrow] (n2) -- node[right, font=\footnotesize, black]{C-D} (n3);

\draw[->, thick, dashed, black] (c1.east) -- (n1.west);
\draw[->, thick, dashed, black] (c2.east) -- (n2.west);
\draw[->, thick, dashed, black] (c3.east) -- (n3.west);

\node[font=\footnotesize, black] at (2.5,-1.6)
  {C-D: Cayley--Dickson doubling};
\end{tikzpicture}
\caption{\emph{The two-column landscape of integer systems over composition algebras. The left column collects the classical crystallographic shells over $\ZZ$, climbing by Cayley--Dickson doubling from the complex examples up to the $E_8$ shell of the Coxeter--Dickson octonions. The right column is the golden analogue obtained by replacing $\ZZ$ with the golden ring $\Zphi$: cyclotomic $H_2$, icosian $H_4$, and icosian-double $H_4\oplus H_4$. The horizontal dashed arrows mark the level at which the lattice condition over $\ZZ$ is replaced by a finite root-shell condition over $\Zphi$.}}
\label{fig:nc-tower}
\end{figure}
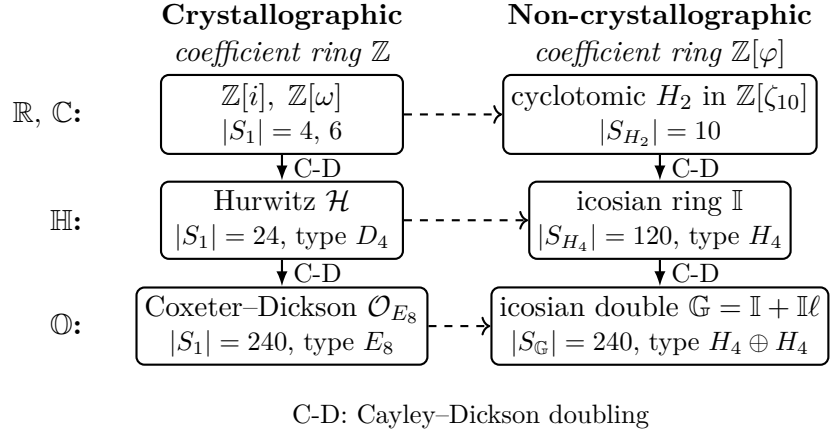

In this work we isolate and formalize an integral root-shell axiomatization for the non-crystallographic case. The construction is deliberately modest: it is not a classification theorem. Rather, it is a framework that recovers the crystallographic examples, explains why \(\Zphi\) replaces \(\ZZ\), accounts for the concrete cases \(H_2\) and \(H_4\), and produces a full-rank octonionic order by doubling the icosian ring. We then prove that the strict overorder version of the stronger G3 problem is closed for this weak double: its norm-polar \(\Zphi\)-dual coincides with the order itself. Under the relaxed trace-integral convention, the direct mixed half-root ansatz and the first \(\sqrt5\)-discriminant tower are also excluded.

\smallskip
\noindent\emph{Relation to existing work.} The icosian ring and its connection with the \(H_4\) root system go back to Tits \cite{Ti80} and Wilson \cite{Wi86}; the role of icosians in the arithmetic of quasicrystals was developed by Moody--Patera \cite{Icosians} and is treated systematically in Conway--Sloane \cite{Conway} and Baake--Grimm \cite{Aperiodic,Ba98}. The non-crystallographic root systems \(H_2,H_3,H_4\) as such are classical \cite{Humphreys}; their integral structure over \(\Zphi\) was studied in Chen--Moody--Patera \cite{CMP}. The present article does not aim to surpass these references in scope or depth; its contribution is twofold. First, it makes precise an axiomatization (Definition~\ref{def:nc-root-shell}) that separates the order, its multiplicative unit object, and the finite root shell, and shows that this is the correct three-tier object once \(\ZZ\) is replaced by \(\Zphi\). Second, it produces a positive rank-\(8\) golden octonion order by Cayley--Dickson doubling of the icosian ring (Proposition~\ref{prop:weak-golden-octonion}) and a sequence of four obstruction theorems (Propositions~\ref{prop:bounded-g3-no-go}--\ref{prop:g3-b-tower-no-go}) limiting how this weak double can be enlarged to a stronger octonionic analogue. 
The present work is structured as follows. In Sec. 2 we introduce orders, trace, norm and root shells. In Sec. 3 we review the crystallographic integer systems over \(\RR,\CC,\HH\) and \(\Oct\). In Sec. 4 we recall root systems, crystallographicity and the correct lattice condition. In Sec. 5 we develop the golden arithmetic of \(\Zphi\). In Sec. 6 we define non-crystallographic root-shell systems and present the \(H_2\), icosian \(H_4\), and icosian-double \(H_4\oplus H_4\) examples, including the strict G3 no-go theorem and the bounded trace-integral G3-B' no-go theorem over the weak double. In Sec. 7 we discuss the interplay with quasicrystals, aperiodic algebras and physical applications. Finally, in Sec. 8 we collect conclusions and future developments. The numerical computation used for the finite enumerations is described after the acknowledgments.

\section{Orders, traces, norms and root shells}

Before defining non-crystallographic integer systems, one has to separate three objects that are often conflated: the full order, its multiplicative unit shell, and the root system carried by a distinguished finite shell. This separation is harmless in the classical examples and indispensable in the golden examples, where the full unit group of \(\Zphi\) is infinite.

\begin{definition}
\label{def:order}
Let \(R\) be a commutative ring, typically \(R=\ZZ\) or \(R=\Zphi\). Let \(A\) be a unital composition algebra over the fraction field of \(R\), equipped with an involution \(x\mapsto \overline{x}\), trace
\begin{equation}
\Tr_A(x)=x+\overline{x},
\label{eq:composition-trace}
\end{equation}
and quadratic norm
\begin{equation}
\Norm_A(x)=x\overline{x}.
\label{eq:composition-norm}
\end{equation}
An \emph{\(R\)-order} in \(A\) is a subset \(\mathcal O\subset A\) such that:
\begin{enumerate}
\item \(\mathcal O\) is a finitely generated \(R\)-module;
\item \(1\in\mathcal O\);
\item \(\mathcal O\) is closed under addition and under every (bracketed) product of finitely many of its elements; in the alternative octonionic case this is equivalent, by Artin's theorem, to closure under all products of any two elements;
\item \(\overline{x}\in\mathcal O\) for every \(x\in\mathcal O\);
\item \(\Tr_A(x),\Norm_A(x)\in R\) for every \(x\in\mathcal O\).
\end{enumerate}
For an alternative algebra the closure clause in (3) is equivalently expressed by requiring that \(\mathcal O\cdot\mathcal O\subset\mathcal O\) in the sense of pointwise products. We refer to \cite{Co46,CoSm,Conway} for the classical examples and to \cite{Di23} for the original formulation.
\end{definition}

In analogy to the octonionic case, the order is generally infinite, whereas the most geometric information is concentrated in finite shells.

\begin{definition}
Let \(\mathcal O\) be an \(R\)-order in \(A\). For \(m\in R\), the \emph{norm shell} of value \(m\) is
\begin{equation}
S_m(\mathcal O)=\{x\in\mathcal O:\Norm_A(x)=m\}.
\label{eq:norm-shell}
\end{equation}
A \emph{root shell} is a finite subset \(S\subset S_m(\mathcal O)\) that carries a root-system structure with respect to the bilinear form
\begin{equation}
\langle x,y\rangle=\Rea(x\overline{y}).
\label{eq:inner-product-algebra}
\end{equation}
\end{definition}

\begin{remark}
The unit group or unit loop of an order is not the same thing as a root shell. In the classical Hurwitz examples the norm-one shell is finite and therefore both notions can coincide. In the golden case, however, \(\Zphi^\times=\{\pm\varphi^k:k\in\ZZ\}\) is infinite, while \(H_2\) and \(H_4\) are finite root systems. This is the central distinction of the paper.
\end{remark}

The following elementary criterion is the computational form of the definition.

\begin{proposition}[Order criterion]
Let \(b_1,\ldots,b_n\) be an \(R\)-basis of a finitely generated \(R\)-module \(\mathcal O\subset A\), with \(b_1=1\). Assume that
\begin{equation}
b_i b_j=\sum_{k=1}^n c_{ij}^k b_k,\qquad c_{ij}^k\in R,
\label{eq:structure-constants}
\end{equation}
and
\begin{equation}
\overline{b_i}=\sum_{k=1}^n d_i^k b_k,\qquad d_i^k\in R.
\label{eq:conjugation-constants}
\end{equation}
Assume moreover that the coordinate polynomials for \(\Tr_A(x)\) and \(\Norm_A(x)\), written in the basis \(b_i\), have coefficients in \(R\). Then \(\mathcal O\) is an \(R\)-order.
\end{proposition}

\begin{proof}
The \(R\)-module structure gives closure under addition. Equation \eqref{eq:structure-constants} gives closure under multiplication, and \eqref{eq:conjugation-constants} gives stability under the involution. The last assumption gives \(\Tr_A(x),\Norm_A(x)\in R\) for a general element \(x=\sum_i r_i b_i\), with \(r_i\in R\). Hence all conditions in the definition of an \(R\)-order are satisfied. \(\square\)
\end{proof}

\section{Crystallographic integer systems over composition algebras}

We now review the classical cases, but in the language of orders and shells. This is not meant as a new classification of integral elements; rather, it is the crystallographic part of the dictionary that will be used later.

\subsection{Gaussian and Eisenstein integers}

Let \(i^2=-1\). The Gaussian order is
\begin{equation}
\ZZ[i]=\{a+bi:a,b\in\ZZ\}\subset\CC.
\label{eq:gaussian-order}
\end{equation}
Its trace and norm are
\begin{equation}
\Tr(a+bi)=2a,\qquad \Norm(a+bi)=a^2+b^2.
\label{eq:gaussian-trace-norm}
\end{equation}
The unit shell is
\begin{equation}
S_1(\ZZ[i])=\{\pm1,\pm i\}.
\label{eq:gaussian-units}
\end{equation}
Geometrically this is the square root shell, which we will consistently denote by \(A_1\oplus A_1\) (the rotation-group label \(C_4\) is used only for the multiplicative unit object).

Let
\begin{equation}
\omega=e^{2\pi i/3},\qquad \omega^2+\omega+1=0.
\label{eq:eisenstein-omega}
\end{equation}
The Eisenstein order is
\begin{equation}
\ZZ[\omega]=\{a+b\omega:a,b\in\ZZ\}.
\label{eq:eisenstein-order}
\end{equation}
The norm-one shell is
\begin{equation}
S_1(\ZZ[\omega])=\{\pm1,\pm\omega,\pm\omega^2\},
\label{eq:eisenstein-units}
\end{equation}
which has cardinality \(6\), not \(3\). With complex multiplication it is the cyclic group \(C_6\), and as a root shell it realizes \(A_2\). The two-dimensional crystallographic shells, together with the non-crystallographic decagonal teaser, are collected in Fig.~\ref{fig:crys-gallery}.

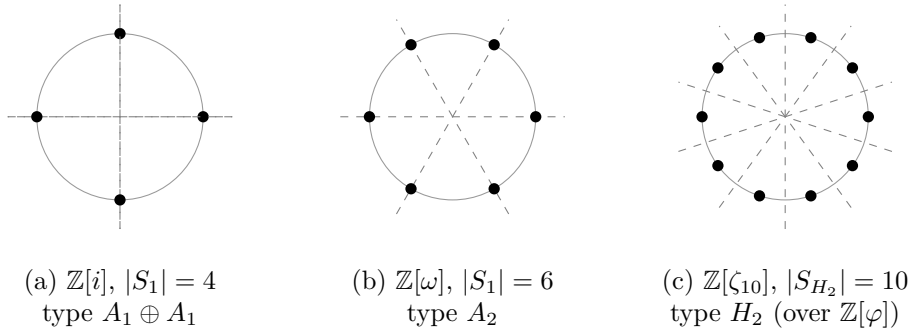
\begin{figure}[htbp]
\centering
\begin{tikzpicture}[
  every node/.style={font=\small},
  scale=1.1,
  v/.style={circle, fill=black, inner sep=1.5pt},
  refl/.style={black!50, dashed, thin},
]
\begin{scope}[xshift=0cm]
\pgfmathsetmacro{\Ra}{1.0}
\draw[black!40, thin] (0,0) circle (\Ra);
\foreach \k in {0,1,2,3} {
  \pgfmathsetmacro{\ang}{\k*90}
  \draw[refl] ({-1.35*cos(\ang)},{-1.35*sin(\ang)})
           -- ({ 1.35*cos(\ang)},{ 1.35*sin(\ang)});
  \node[v] at ({\Ra*cos(\ang)},{\Ra*sin(\ang)}) {};
}
\node[anchor=north, font=\footnotesize, align=center] at (0,-1.7)
  {(a) $\ZZ[i]$, $|S_{1}|=4$\\type $A_{1}\oplus A_{1}$};
\end{scope}

\begin{scope}[xshift=4.0cm]
\pgfmathsetmacro{\Rb}{1.0}
\draw[black!40, thin] (0,0) circle (\Rb);
\foreach \k in {0,1,2} {
  \pgfmathsetmacro{\ang}{\k*60}
  \draw[refl] ({-1.35*cos(\ang)},{-1.35*sin(\ang)})
           -- ({ 1.35*cos(\ang)},{ 1.35*sin(\ang)});
}
\foreach \k in {0,...,5} {
  \pgfmathsetmacro{\ang}{\k*60}
  \node[v] at ({\Rb*cos(\ang)},{\Rb*sin(\ang)}) {};
}
\node[anchor=north, font=\footnotesize, align=center] at (0,-1.7)
  {(b) $\ZZ[\omega]$, $|S_{1}|=6$\\type $A_{2}$};
\end{scope}

\begin{scope}[xshift=8.0cm]
\pgfmathsetmacro{\Rc}{1.0}
\draw[black!40, thin] (0,0) circle (\Rc);
\foreach \k in {0,1,2,3,4} {
  \pgfmathsetmacro{\ang}{\k*36+18}
  \draw[refl] ({-1.35*cos(\ang)},{-1.35*sin(\ang)})
           -- ({ 1.35*cos(\ang)},{ 1.35*sin(\ang)});
}
\foreach \k in {0,...,9} {
  \pgfmathsetmacro{\ang}{\k*36}
  \node[v] at ({\Rc*cos(\ang)},{\Rc*sin(\ang)}) {};
}
\node[anchor=north, font=\footnotesize, align=center] at (0,-1.7)
  {(c) $\ZZ[\zeta_{10}]$, $|S_{H_{2}}|=10$\\type $H_{2}$ (over $\Zphi$)};
\end{scope}
\end{tikzpicture}
\caption{\emph{Two-dimensional unit shells of the smallest composition-algebra orders. Panel (a) shows the four Gaussian units forming the reducible crystallographic root system $A_{1}\oplus A_{1}$; panel (b) shows the six Eisenstein units forming the $A_{2}$ shell. Panel (c) anticipates the non-crystallographic case: the ten roots of $H_{2}$ realized inside $\ZZ[\zeta_{10}]$, with Cartan coefficients that lie in $\Zphi$ but not in $\ZZ$.}}
\label{fig:crys-gallery}
\end{figure}

\begin{remark}
If one keeps only \(\{1,\omega,\omega^2\}\), one obtains the oriented rotation subgroup \(C_3\), not the full \(A_2\) root system.
\end{remark}

\subsection{Hamilton, hybrid and Hurwitz quaternions}

Let \(\HH=\RR\langle i,j,k\rangle\), with \(i^2=j^2=k^2=ijk=-1\). The Hamilton order is the \(\ZZ\)-span
\begin{equation}
\mathcal H_0=\ZZ[1,i,j,k],
\label{eq:hamilton-order}
\end{equation}
which is an order but not the maximal order of \(\HH\) (the maximal order is the Hurwitz order \(\mathcal H\) introduced below).
For \(x=a+bi+cj+dk\) one has
\begin{equation}
\Tr(x)=2a,\qquad \Norm(x)=a^2+b^2+c^2+d^2.
\label{eq:hamilton-trace-norm}
\end{equation}
Its norm-one shell is
\begin{equation}
\{\pm1,\pm i,\pm j,\pm k\},
\label{eq:hamilton-units}
\end{equation}
which is the quaternion group \(Q_8\).

The Hurwitz order can be written as
\begin{equation}
\mathcal H=\ZZ[1,u,v,w],
\label{eq:hurwitz-order}
\end{equation}
where
\begin{align}
u&=\frac{1}{2}(1-i-j+k),&
v&=\frac{1}{2}(1+i-j-k),&
w&=\frac{1}{2}(1-i+j-k).
\label{eq:hurwitz-basis}
\end{align}
Its \(24\) norm-one elements are
\begin{equation}
\{\pm1,\pm i,\pm j,\pm k\}\cup
\left\{\frac{1}{2}(\pm1\pm i\pm j\pm k)\right\},
\label{eq:hurwitz-units}
\end{equation}
with all choices of signs in the second set. They form the binary tetrahedral group \(2T\) and the \(D_4\) root shell.

Finally, the hybrid quaternionic order is naturally written as
\begin{equation}
\ZZ[\omega]\oplus \ZZ[\omega]j.
\label{eq:hybrid-order}
\end{equation}
Its \(12\)-element unit shell is
\begin{equation}
\{\pm1,\pm\omega,\pm\omega^2,\pm j,\pm\omega j,\pm\omega^2j\}.
\label{eq:hybrid-units}
\end{equation}
This shell realizes \(A_2\oplus A_2\). With quaternionic multiplication it is the dicyclic group \(\operatorname{Dic}_3\), equivalently the semidirect form in which an order-four element acts by inversion on the Eisenstein rotations.

\subsection{Graves--Cayley and Coxeter--Dickson octonions}

The octonionic examples require more care because \(\Oct\) is not associative. Thus the finite multiplicative shells are not groups but Moufang loops. This is perfectly analogous to the way in which the octonionic unit sphere is multiplicative but not associative.

Let
\begin{equation}
\{1,i,j,k,l,il,jl,kl\}
\label{eq:octonion-dickson-basis}
\end{equation}
be a Dickson basis of \(\Oct\), with multiplication fixed by the Fano-plane convention. The Graves--Cayley order is the \(\ZZ\)-span of this basis. Its norm-one shell is
\begin{equation}
\{\pm1,\pm i,\pm j,\pm k,\pm l,\pm il,\pm jl,\pm kl\}.
\label{eq:cayley-graves-units}
\end{equation}
The presence of \(-1\) is essential; without it the shell is not closed under taking opposites and cannot be a root system.

Coxeter's integral Cayley numbers \cite{Co46} give octonionic orders whose minimal shells have \(240\) elements and realize the \(E_8\) root system. In one Dickson convention this order can be generated by replacing one basis element with a half-sum such as
\begin{equation}
h=\frac{i+j+k+l}{2},
\label{eq:coxeter-half-sum}
\end{equation}
which has \(\Norm(h)=1\) and therefore belongs to the unit shell of the order.
The exact list of the \(240\) elements is best treated as a supplementary table or appendix, because the mathematical point is that the finite shell is the \(E_8\) shell and the multiplication gives a Moufang loop. The Petrie projection of $E_8$, displayed in Fig.~\ref{fig:e8-petrie}, illustrates the thirty-fold Coxeter symmetry that connects $E_8$ to the icosian shell of Sec.~6.

\begin{figure}[htbp]
\centering
\begin{tikzpicture}[
  every node/.style={font=\small},
  scale=2.4,
  v/.style={circle, fill=black, inner sep=0.65pt},
  vouter/.style={circle, fill=black, inner sep=0.9pt},
  e/.style={black!55, very thin, opacity=0.55},
  hilight/.style={black, thin},
]
\pgfmathsetmacro{\RR}{1.95}
\def\rA{0.20*\RR}
\def\rB{0.36*\RR}
\def\rC{0.50*\RR}
\def\rD{0.62*\RR}
\def\rE{0.74*\RR}
\def\rF{0.84*\RR}
\def\rG{0.93*\RR}
\def\rH{1.00*\RR}

\foreach \rr in {\rA,\rB,\rC,\rD,\rE,\rF,\rG,\rH} {
  \draw[black!18, very thin] (0,0) circle (\rr);
}

\foreach \k in {0,...,29} {
  \pgfmathsetmacro{\angA}{\k*12+0}
  \pgfmathsetmacro{\angB}{\k*12+6}
  \pgfmathsetmacro{\angC}{\k*12+0}
  \pgfmathsetmacro{\angD}{\k*12+6}
  \pgfmathsetmacro{\angE}{\k*12+0}
  \pgfmathsetmacro{\angF}{\k*12+6}
  \pgfmathsetmacro{\angG}{\k*12+0}
  \pgfmathsetmacro{\angH}{\k*12+6}
  \coordinate (Va\k) at ({\rA*cos(\angA)},{\rA*sin(\angA)});
  \coordinate (Vb\k) at ({\rB*cos(\angB)},{\rB*sin(\angB)});
  \coordinate (Vc\k) at ({\rC*cos(\angC)},{\rC*sin(\angC)});
  \coordinate (Vd\k) at ({\rD*cos(\angD)},{\rD*sin(\angD)});
  \coordinate (Ve\k) at ({\rE*cos(\angE)},{\rE*sin(\angE)});
  \coordinate (Vf\k) at ({\rF*cos(\angF)},{\rF*sin(\angF)});
  \coordinate (Vg\k) at ({\rG*cos(\angG)},{\rG*sin(\angG)});
  \coordinate (Vh\k) at ({\rH*cos(\angH)},{\rH*sin(\angH)});
}

\foreach \k in {0,...,29} {
  \pgfmathtruncatemacro{\next}{mod(\k+1,30)}
  \draw[e] (Vh\k) -- (Vh\next);
  \draw[e] (Vh\k) -- (Vg\k);
  \draw[e] (Vh\k) -- (Vg\next);
  \draw[e] (Vg\k) -- (Vf\k);
  \draw[e] (Vg\k) -- (Vf\next);
  \draw[e] (Vf\k) -- (Ve\k);
  \draw[e] (Vf\k) -- (Ve\next);
  \draw[e] (Ve\k) -- (Vd\k);
  \draw[e] (Vd\k) -- (Vc\k);
  \draw[e] (Vc\k) -- (Vb\k);
  \draw[e] (Vb\k) -- (Va\k);
}

\foreach \k in {0,...,29} {
  \pgfmathtruncatemacro{\nn}{mod(\k+11,30)}
  \draw[hilight] (Vh\k) -- (Vh\nn);
}

\foreach \k in {0,...,29} {
  \node[v] at (Va\k) {};
  \node[v] at (Vb\k) {};
  \node[v] at (Vc\k) {};
  \node[v] at (Vd\k) {};
  \node[v] at (Ve\k) {};
  \node[v] at (Vf\k) {};
  \node[v] at (Vg\k) {};
  \node[vouter] at (Vh\k) {};
}
\end{tikzpicture}
\caption{\emph{Petrie projection of the $E_{8}$ root system in the Coxeter plane. The $240$ roots fall into eight concentric $30$-gons of $30$ vertices each, and the highlighted $\{30/11\}$ star polygon on the outer ring traces the orbit of a single root under the Coxeter element. In the present paper this picture is the crystallographic terminus on the left column of Fig.~\ref{fig:nc-tower}: replacing $\ZZ$ with $\Zphi$ inside the icosian ring recovers the non-crystallographic shell $H_{4}\oplus H_{4}$ of cardinality $240$ shown in Fig.~\ref{fig:h4-h4-decomposition}.}}
\label{fig:e8-petrie}
\end{figure}
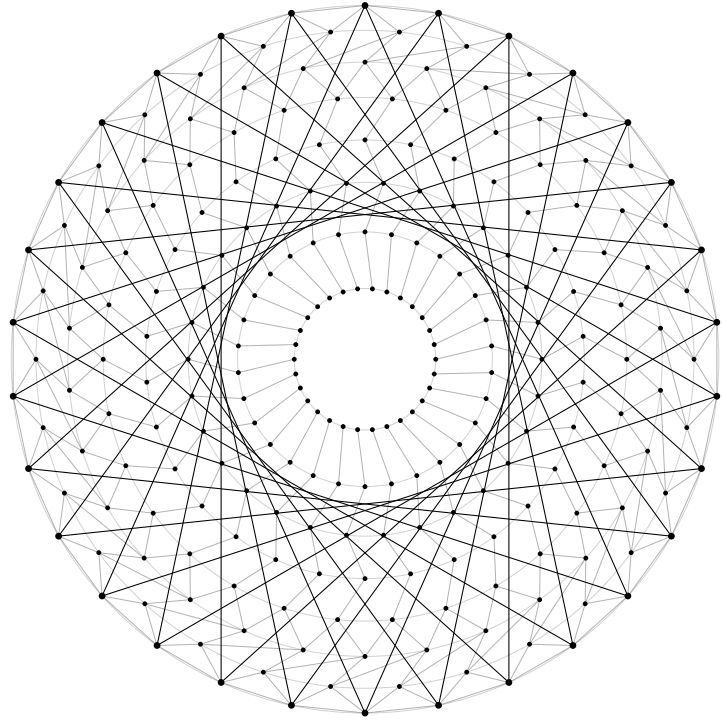

\begin{remark}
The octonionic review also explains why the word ``order'' must be used with some care in non-associative algebras. Closure under multiplication is still meaningful, and alternative laws are enough for the standard inverse and Moufang-loop statements, but one must not silently import associative ring language.
\end{remark}

\section{Root systems, lattices and crystallographicity}

In order to define a non-crystallographic analogue of integral elements, we need the correct crystallographic condition. This is the point where the ordinary lattice condition over \(\ZZ\) must be replaced by a golden integrality condition.

\begin{definition}
Let \(E\) be a Euclidean space with inner product \(\langle\cdot,\cdot\rangle\). A finite subset \(\Phi\subset E\setminus\{0\}\) is a \emph{reduced root system} if
\begin{equation}
\Phi\cap\RR\alpha=\{\pm\alpha\}
\label{eq:root-line-condition}
\end{equation}
for every \(\alpha\in\Phi\), and if for every \(\alpha\in\Phi\) the reflection
\begin{equation}
r_\alpha(v)=v-\frac{2\langle v,\alpha\rangle}{\langle \alpha,\alpha\rangle}\alpha
\label{eq:root-reflection}
\end{equation}
preserves \(\Phi\). All root systems considered below are reduced.
\end{definition}

The group generated by the reflections \(r_\alpha\) is a Weyl group in the crystallographic cases and, more generally, a Coxeter reflection group \cite{Humphreys}. For \(H_2,H_3,H_4\) we will use the term Coxeter group.

\begin{definition}
A reduced root system \(\Phi\) is \emph{crystallographic} if
\begin{equation}
\frac{2\langle \beta,\alpha\rangle}{\langle \alpha,\alpha\rangle}\in\ZZ
\label{eq:crystallographic-condition}
\end{equation}
for every \(\alpha,\beta\in\Phi\). Its root lattice is
\begin{equation}
\Lambda(\Phi)=\Span_{\ZZ}(\Phi)
=\left\{\sum_{i=1}^r n_i\alpha_i:n_i\in\ZZ,\ \alpha_i\in\Phi\right\}.
\label{eq:root-lattice}
\end{equation}
\end{definition}

\begin{remark}
The factor \(2\) in \eqref{eq:crystallographic-condition} is not cosmetic. For \(A_2\), two simple roots of equal length satisfy
\begin{equation}
\frac{\langle\alpha,\beta\rangle}{\langle\alpha,\alpha\rangle}=-\frac12,
\qquad
\frac{2\langle\alpha,\beta\rangle}{\langle\alpha,\alpha\rangle}=-1.
\label{eq:a2-factor-two}
\end{equation}
Without this factor the hexagonal root system would incorrectly fail to be crystallographic.
\end{remark}

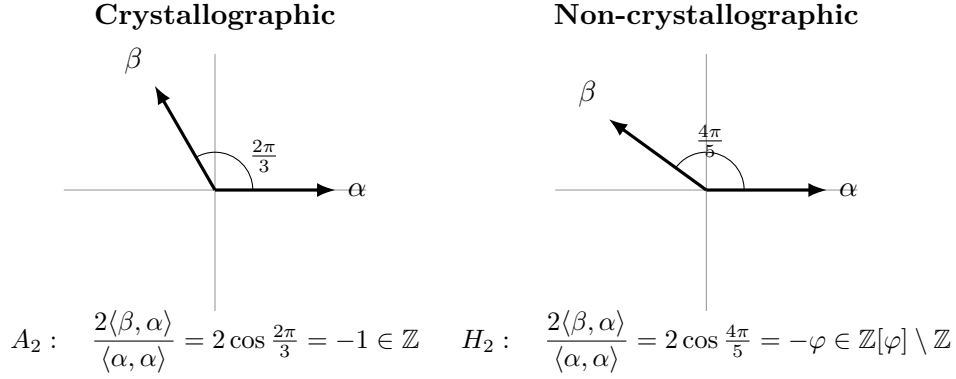
\begin{figure}[htbp]
\centering
\begin{tikzpicture}[
  every node/.style={font=\small},
  rootarrow/.style={-{Latex[length=2.5mm]}, very thick, black},
  axisline/.style={black!40, thin},
]
\begin{scope}[xshift=0cm]
\draw[axisline] (-2.0,0) -- (2.0,0);
\draw[axisline] (0,-1.6) -- (0,1.8);

\coordinate (Oa) at (0,0);
\coordinate (alphaA) at (1.6,0);
\coordinate (betaA) at ({1.6*cos(120)}, {1.6*sin(120)});

\draw[rootarrow] (Oa) -- (alphaA) node[right, black] {$\alpha$};
\draw[rootarrow] (Oa) -- (betaA) node[above left, black] {$\beta$};

\draw[thin, black] (0.5,0) arc[start angle=0, end angle=120, radius=0.5];
\node[black] at (0.65,0.45) {\footnotesize $\tfrac{2\pi}{3}$};

\node[align=center, font=\footnotesize, black] at (0,-2.0)
  {$A_2:\quad \dfrac{2\langle\beta,\alpha\rangle}{\langle\alpha,\alpha\rangle}=2\cos\tfrac{2\pi}{3}=-1\in\ZZ$};
\node[align=center, font=\small\bfseries, black] at (0,2.3) {Crystallographic};
\end{scope}

\begin{scope}[xshift=6.5cm]
\draw[axisline] (-2.0,0) -- (2.0,0);
\draw[axisline] (0,-1.6) -- (0,1.8);

\coordinate (Oh) at (0,0);
\coordinate (alphaH) at (1.6,0);
\coordinate (betaH) at ({1.6*cos(144)}, {1.6*sin(144)});

\draw[rootarrow] (Oh) -- (alphaH) node[right, black] {$\alpha$};
\draw[rootarrow] (Oh) -- (betaH) node[above left, black] {$\beta$};

\draw[thin, black] (0.5,0) arc[start angle=0, end angle=144, radius=0.5];
\node[black] at (0.05,0.65) {\footnotesize $\tfrac{4\pi}{5}$};

\node[align=center, font=\footnotesize, black] at (0,-2.0)
  {$H_2:\quad \dfrac{2\langle\beta,\alpha\rangle}{\langle\alpha,\alpha\rangle}=2\cos\tfrac{4\pi}{5}=-\varphi\in\Zphi\setminus\ZZ$};
\node[align=center, font=\small\bfseries, black] at (0,2.3) {Non-crystallographic};
\end{scope}
\end{tikzpicture}
\caption{\emph{Cartan-coefficient comparison between the smallest crystallographic and non-crystallographic rank-two root systems. On the left, the $A_2$ simple roots subtend $2\pi/3$ and produce the integer Cartan coefficient $-1$. On the right, the $H_2$ simple roots subtend $4\pi/5$ and produce the golden Cartan coefficient $-\varphi$, which lies in $\Zphi$ but not in $\ZZ$. This is the obstruction that forces the replacement of $\ZZ$ by $\Zphi$ as the natural coefficient ring.}}
\label{fig:cartan-a2-vs-h2}
\end{figure}

The non-crystallographic obstruction is not that some root angle is an irrational multiple of \(\pi\): the angles of \(H_2=I_2(5)\) are multiples of \(\pi/5\). Rather, the obstruction lies in the Cartan coefficients themselves: they need not be in \(\ZZ\), but they lie in \(\Zphi\). Consequently the ordinary \(\ZZ\)-span of the root system can have rational rank larger than the real dimension and may become dense in the underlying Euclidean space. The correct replacement is therefore not a lattice over \(\ZZ\) but a finite root shell over the algebraic integer ring \(\Zphi\).

The smallest non-crystallographic example is
\begin{equation}
H_2=\left\{\left(\cos\frac{k\pi}{5},\sin\frac{k\pi}{5}\right):0\leq k<10\right\}.
\label{eq:h2-decagon}
\end{equation}
The three-dimensional example \(H_3\) has \(30\) roots,
\begin{align}
H_3={}&
\left\{(\pm1,0,0)\text{ and all permutations}\right\}
\notag\\
&\cup
\left\{\frac12(\pm1,\pm\varphi,\pm\varphi^{-1})\text{ and all even permutations}\right\}.
\label{eq:h3-roots}
\end{align}
Geometrically these \(30\) roots are the vertices of the icosidodecahedron, not the vertices of the icosahedron.

Finally, \(H_4\) is realized by the \(120\) roots of the \(600\)-cell:
\begin{align}
H_4={}&
\left\{(\pm1,0,0,0)\text{ and all permutations}\right\}
\notag\\
&\cup
\left\{\frac12(\pm1,\pm1,\pm1,\pm1)\right\}
\notag\\
&\cup
\left\{\frac12(0,\pm1,\pm\varphi,\pm\varphi^{-1})\text{ and all even permutations}\right\}.
\label{eq:h4-roots}
\end{align}

The $120$-vertex realization of $H_4$ as the $600$-cell is visualized in the Petrie projection of Fig.~\ref{fig:600cell-petrie}.

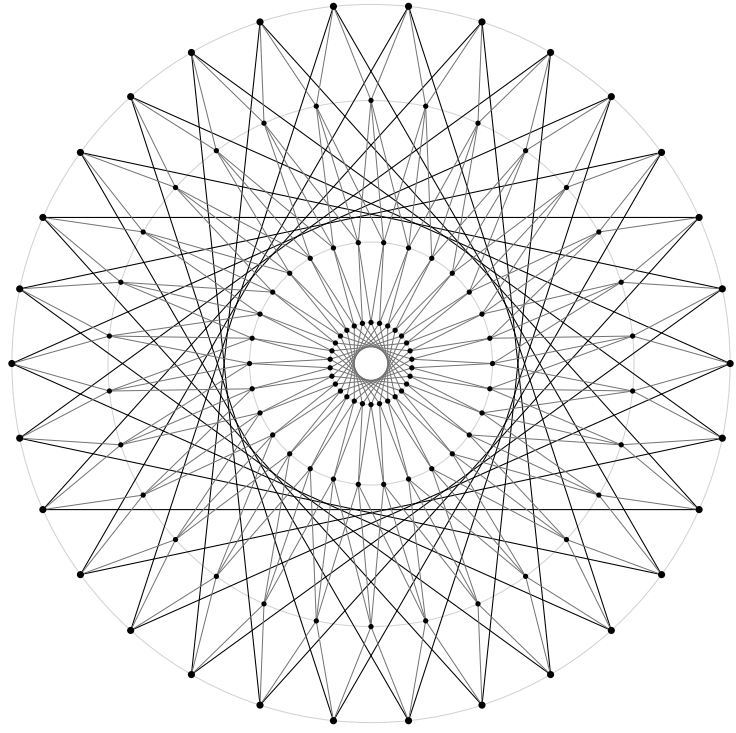
\begin{figure}[htbp]
\centering
\begin{tikzpicture}[
  every node/.style={font=\small},
  scale=2.6,
  v/.style={circle, fill=black, inner sep=0.7pt},
  vouter/.style={circle, fill=black, inner sep=0.95pt},
  e/.style={black!55, very thin},
  starchord/.style={black, thin},
]
\pgfmathsetmacro{\RA}{2*sin(11*180/30)}
\pgfmathsetmacro{\RB}{2*sin( 7*180/30)}
\pgfmathsetmacro{\RC}{2*sin( 3*180/30)}
\pgfmathsetmacro{\RD}{2*sin( 1*180/30)}

\draw[black!18, very thin] (0,0) circle (\RA);
\draw[black!18, very thin] (0,0) circle (\RB);
\draw[black!18, very thin] (0,0) circle (\RC);
\draw[black!18, very thin] (0,0) circle (\RD);

\foreach \k in {0,...,29} {
  \pgfmathsetmacro{\ang}{\k*12}
  \coordinate (A\k) at ({\RA*cos(\ang)},{\RA*sin(\ang)});
}
\foreach \k in {0,...,29} {
  \pgfmathsetmacro{\ang}{\k*12+6}
  \coordinate (B\k) at ({\RB*cos(\ang)},{\RB*sin(\ang)});
}
\foreach \k in {0,...,29} {
  \pgfmathsetmacro{\ang}{\k*12}
  \coordinate (C\k) at ({\RC*cos(\ang)},{\RC*sin(\ang)});
}
\foreach \k in {0,...,29} {
  \pgfmathsetmacro{\ang}{\k*12+6}
  \coordinate (D\k) at ({\RD*cos(\ang)},{\RD*sin(\ang)});
}

\foreach \k in {0,...,29} {
  \pgfmathtruncatemacro{\next}{mod(\k+11,30)}
  \draw[starchord] (A\k) -- (A\next);
}

\foreach \k in {0,...,29} {
  \pgfmathtruncatemacro{\prevk}{mod(\k+29,30)}
  \draw[e] (A\k) -- (B\k);
  \draw[e] (A\k) -- (B\prevk);
}
\foreach \k in {0,...,29} {
  \pgfmathtruncatemacro{\prevk}{mod(\k+29,30)}
  \draw[e] (B\k) -- (C\k);
  \draw[e] (B\k) -- (C\prevk);
  \pgfmathtruncatemacro{\nextk}{mod(\k+1,30)}
  \draw[e] (B\k) -- (C\nextk);
}
\foreach \k in {0,...,29} {
  \pgfmathtruncatemacro{\prevk}{mod(\k+29,30)}
  \draw[e] (C\k) -- (D\k);
  \draw[e] (C\k) -- (D\prevk);
}
\foreach \k in {0,...,29} {
  \pgfmathtruncatemacro{\nextk}{mod(\k+11,30)}
  \draw[e] (D\k) -- (D\nextk);
}

\foreach \k in {0,...,29} {
  \node[vouter] at (A\k) {};
  \node[v]      at (B\k) {};
  \node[v]      at (C\k) {};
  \node[v]      at (D\k) {};
}
\end{tikzpicture}
\caption{\emph{Petrie projection of the $600$-cell, whose $120$ vertices realize the root shell $S_{H_4}$. The vertices fall into four concentric $30$-gons whose radii are $2\sin(m\pi/30)$ for $m\in\{11,7,3,1\}$; the highlighted chords trace the $\{30/11\}$ Petrie star polygon. The thirty-fold rotational symmetry visible in the picture is the projection of the Coxeter element of $H_4$ acting on the $600$-cell.}}
\label{fig:600cell-petrie}
\end{figure}

\section{Golden arithmetic and Dirichlet integers}

We now isolate the arithmetic that replaces \(\ZZ\) in the non-crystallographic examples. The relevant field is not mysterious: it is the real quadratic field generated by the golden ratio.

Let
\begin{equation}
K=\KK,\qquad \varphi=\frac{1+\sqrt5}{2}.
\label{eq:golden-field}
\end{equation}
The nontrivial Galois conjugation is denoted by \(x\mapsto x^*\), and it sends
\begin{equation}
\varphi^*=1-\varphi=-\varphi^{-1}.
\label{eq:phi-conjugate}
\end{equation}

\begin{proposition}
The ring of integers of \(K\) is
\begin{equation}
\mathcal O_K=\Zphi.
\label{eq:ring-of-integers-golden}
\end{equation}
For \(x=a+b\varphi\), with \(a,b\in\ZZ\), one has
\begin{equation}
x^*=a+b(1-\varphi),
\label{eq:golden-conjugation}
\end{equation}
\begin{equation}
\Tr_{K/\QQ}(x)=x+x^*=2a+b,
\label{eq:golden-trace}
\end{equation}
and
\begin{equation}
\Norm_{K/\QQ}(x)=xx^*=a^2+ab-b^2.
\label{eq:golden-field-norm}
\end{equation}
Moreover
\begin{equation}
\mathcal O_K^\times=\{\pm\varphi^k:k\in\ZZ\}.
\label{eq:golden-units}
\end{equation}
\end{proposition}

\begin{proof}
Since \(5\equiv1\pmod4\), the ring of integers of \(\QQ(\sqrt5)\) is \(\ZZ[(1+\sqrt5)/2]\); see, e.g., Marcus, \emph{Number Fields}, Ch.~2 or Serre \cite{Serre}. Equations \eqref{eq:golden-conjugation}, \eqref{eq:golden-trace} and \eqref{eq:golden-field-norm} follow by direct multiplication using \(\varphi^2=\varphi+1\). The unit assertion is Dirichlet's unit theorem for the real-quadratic field \(K\), of free rank one, with \(\varphi\) a fundamental unit and \(\Norm_{K/\QQ}(\varphi)=-1\). \(\square\)
\end{proof}

A useful geometric picture of \(\Zphi\) is obtained by the Minkowski embedding \(a+b\varphi\mapsto(a+b\varphi,a+b\varphi^{*})\), shown in Fig.~\ref{fig:zphi-minkowski}. The ring is dense in \(\RR\) along the first coordinate but discrete as a rank-two lattice in \(\RR^{2}\).

\begin{figure}[htbp]
\centering
\begin{tikzpicture}[
  every node/.style={font=\small},
  scale=0.78,
  lattpt/.style={circle, fill=black, inner sep=1.2pt},
  axisline/.style={->, thick, black},
]
\pgfmathsetmacro{\PHI}{(1+sqrt(5))/2}
\pgfmathsetmacro{\PHIS}{1-\PHI}

\draw[axisline] (-4.6,0) -- (4.8,0) node[right, black] {$\sigma_1: x \mapsto x$};
\draw[axisline] (0,-3.6) -- (0,3.8) node[above, black] {$\sigma_2: x \mapsto x^{*}$};

\foreach \a in {-3,-2,-1,0,1,2,3} {
  \foreach \b in {-3,-2,-1,0,1,2,3} {
    \pgfmathsetmacro{\xx}{\a + \b*\PHI}
    \pgfmathsetmacro{\yy}{\a + \b*\PHIS}
    \ifdim \xx pt > -4.6pt
    \ifdim \xx pt < 4.6pt
    \ifdim \yy pt > -3.6pt
    \ifdim \yy pt < 3.6pt
      \node[lattpt] at (\xx,\yy) {};
    \fi\fi\fi\fi
  }
}

\pgfmathsetmacro{\onex}{1}
\pgfmathsetmacro{\oney}{1}
\pgfmathsetmacro{\phix}{\PHI}
\pgfmathsetmacro{\phiy}{\PHIS}

\draw[->, thick, black] (0,0) -- (\onex,\oney);
\draw[->, thick, black] (0,0) -- (\phix,\phiy);

\node[anchor=south west, font=\footnotesize, black] at (\onex,\oney) {$1$};
\node[anchor=north east, font=\footnotesize, black] at (\phix,\phiy) {$\varphi$};

\draw[<->, thin, black!60, dashed]
  (-4.2,-0.07) -- (4.6,-0.07);
\node[font=\footnotesize, anchor=west, black] at (4.7,-0.3)
  {dense in $\RR$ along $\sigma_1$};

\fill[black!8]
  (0,0) -- (1,1) -- (1+\PHI,1+\PHIS) -- (\PHI,\PHIS) -- cycle;
\draw[black!55, thin]
  (0,0) -- (1,1) -- (1+\PHI,1+\PHIS) -- (\PHI,\PHIS) -- cycle;

\end{tikzpicture}
\caption{\emph{The ring $\Zphi$ embedded as a rank-two lattice in $\RR^2$ via the Minkowski map $a+b\varphi\mapsto(a+b\varphi,\,a+b\varphi^{*})$. The two arrows are the images of the $\ZZ$-basis $\{1,\varphi\}$; the shaded parallelogram is a fundamental cell. Projection on the first coordinate $\sigma_1$ recovers the dense subset $\Zphi\subset\RR$, while the lattice itself is discrete in $\RR^2$. This duality, governed by the Galois action $\varphi^{*}=1-\varphi$, is the arithmetic engine behind the cut-and-project description of decagonal and icosahedral quasicrystals.}}
\label{fig:zphi-minkowski}
\end{figure}
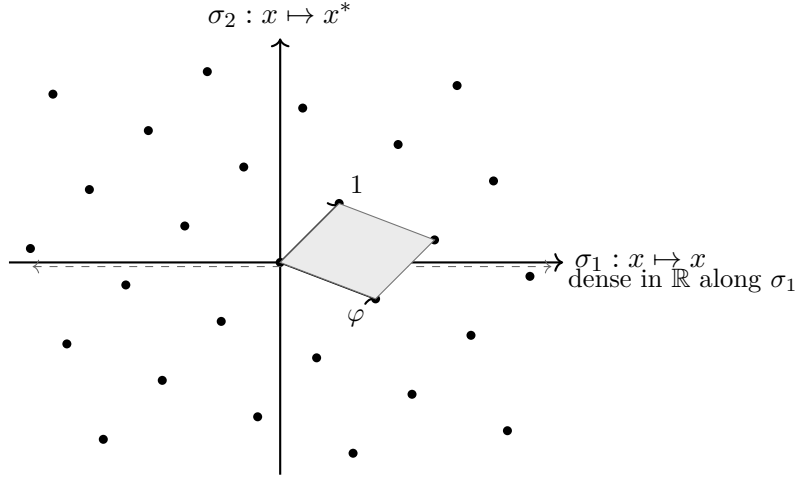

There is one more map which appears in the icosian literature, especially in Tits' construction \cite{Ti80}. It is useful, but it is not a field norm.

\begin{definition}
The \emph{Dirichlet height}, or \emph{Tits projection}, is the map
\begin{equation}
n_\varphi(a+b\varphi)=a.
\label{eq:tits-height}
\end{equation}
Equivalently,
\begin{equation}
n_\varphi(x)=\kappa x+(\kappa x)^*,
\qquad
\kappa=\frac{2}{5+\sqrt5}=\frac{3-\varphi}{5}.
\label{eq:tits-height-kappa}
\end{equation}
\end{definition}

\begin{remark}
The equivalence of \eqref{eq:tits-height} and \eqref{eq:tits-height-kappa} is a one-line check. With \(\kappa=(3-\varphi)/5\) one has \(\kappa^*=(2+\varphi)/5\). For \(x=a+b\varphi\) we obtain \(\kappa x=(3a-b+(2b-a)\varphi)/5\) and \(\kappa^* x^*=(2a+b+(a-2b)\varphi)/5\). Their sum is \(a\), confirming \(n_\varphi(x)=a\).
\end{remark}

\begin{remark}
The positive composition norm on \(K\subset\RR\), the field norm \(\Norm_{K/\QQ}\), and the Dirichlet height \(n_\varphi\) are different objects. The field norm is multiplicative but indefinite; the Dirichlet height is a projection; the positive square \(x^2\) is not the algebraic field norm.
\end{remark}

\section{Non-crystallographic root-shell systems}

We can now state the replacement for Johnson's lattice-based definition. The guiding analogy is simple: in the crystallographic case, the finite shell carries a root system over \(\ZZ\); in the golden case, the finite shell carries a root system whose Cartan coefficients lie in \(\Zphi\).

\begin{definition}
\label{def:nc-root-shell}
Let \(R=\Zphi\), and let \(A\) be a composition algebra defined over \(K=\KK\), equipped with an involution and a quadratic norm \(\Norm_A\). Fix the real embedding \(K\hookrightarrow\RR\) sending \(\varphi\) to \((1+\sqrt5)/2\), so that \(A\otimes_K\RR\) inherits a Euclidean inner product
\begin{equation}
\langle x,y\rangle_{\RR}=\Rea(x\overline{y}).
\label{eq:real-inner-product}
\end{equation}
A \emph{non-crystallographic integral root-shell system} is a pair \((\mathcal O,S)\) such that:
\begin{enumerate}
\item \(\mathcal O\) is an \(R\)-order in \(A\);
\item \(S\subset \mathcal O\) is finite;
\item \(S=-S\);
\item \(S\) is contained in a fixed norm shell \(S_m(\mathcal O)\);
\item for every \(\alpha\in S\), the reflection
\begin{equation}
r_\alpha\colon v\longmapsto v-\frac{2\langle v,\alpha\rangle_{\RR}}{\langle\alpha,\alpha\rangle_{\RR}}\alpha,
\qquad v\in A\otimes_K\RR,
\label{eq:nc-reflection}
\end{equation}
maps \(S\) to itself;
\item for all \(\alpha,\beta\in S\),
\begin{equation}
\frac{2\langle \beta,\alpha\rangle_{\RR}}{\langle\alpha,\alpha\rangle_{\RR}}\in\Zphi.
\label{eq:golden-cartan}
\end{equation}
\end{enumerate}
\end{definition}

\subsection{\texorpdfstring{\(H_2\)}{H2} and cyclotomic integers}

Let \(\zeta_{10}=e^{\pi i/5}\). The decagonal shell
\begin{equation}
S_{H_2}=\{\zeta_{10}^k:0\leq k<10\}\subset\CC
\label{eq:h2-cyclotomic-shell}
\end{equation}
is naturally contained in the cyclotomic order \(\ZZ[\zeta_{10}]\). Its real inner products lie in the maximal real subfield, whose ring of integers is \(\Zphi\).

\begin{proposition}
\label{prop:h2-cartan}
The shell \(S_{H_2}\) is a non-crystallographic integral root shell over \(\Zphi\).
\end{proposition}

\begin{proof}
The set \(S_{H_2}=\{\zeta_{10}^k:0\leq k<10\}\) is finite and satisfies \(S_{H_2}=-S_{H_2}\) because \(\zeta_{10}^{k+5}=-\zeta_{10}^k\). The reflections in the lines orthogonal to its elements generate the dihedral Coxeter group \(I_2(5)\) of order \(10\), hence preserve \(S_{H_2}\). For the Cartan coefficients we use the fact that the angle between any two roots of \(I_2(5)\) is a multiple of \(\pi/5\); the only inner-product values that occur are
\begin{equation}
2\cos\frac{k\pi}{5}\in\{2,\;\varphi,\;\varphi-1,\;-\varphi+1,\;-\varphi,\;-2\},\qquad k=0,1,\ldots,5,
\end{equation}
and \(\varphi,\;\varphi-1=\varphi^{-1}\in\Zphi\). Therefore
\begin{equation}
\frac{2\langle\beta,\alpha\rangle_{\RR}}{\langle\alpha,\alpha\rangle_{\RR}}\in\Zphi
\qquad\text{for every pair }\alpha,\beta\in S_{H_2}.
\label{eq:h2-cartan}
\end{equation}
Not all values lie in \(\ZZ\), since \(\varphi\notin\ZZ\); this is the ordinary non-crystallographicity. \(\square\)
\end{proof}

The decagonal shell and its reflection geometry are summarized in Fig.~\ref{fig:h2-decagon}.

\begin{figure}[htbp]
\centering
\begin{tikzpicture}[
  every node/.style={font=\small},
  scale=1.3,
  root/.style={circle, fill=black, inner sep=1.6pt},
  rline/.style={very thick, black, -{Latex[length=2.5mm]}},
  refl/.style={black!55, dashed, thin},
]
\pgfmathsetmacro{\R}{1.7}

\foreach \k in {0,1,2,3,4} {
  \pgfmathsetmacro{\ang}{\k*36+18}
  \draw[refl] ({-(\R+0.4)*cos(\ang)},{-(\R+0.4)*sin(\ang)})
            -- ({ (\R+0.4)*cos(\ang)},{ (\R+0.4)*sin(\ang)});
}

\draw[black!30, thin] (0,0) circle (\R);

\foreach \k in {0,...,9} {
  \pgfmathsetmacro{\ang}{\k*36}
  \node[root] (r\k) at ({\R*cos(\ang)},{\R*sin(\ang)}) {};
}

\pgfmathsetmacro{\angA}{0}
\pgfmathsetmacro{\angB}{144}

\draw[rline] (0,0) -- ({\R*cos(\angA)},{\R*sin(\angA)});
\draw[rline] (0,0) -- ({\R*cos(\angB)},{\R*sin(\angB)});

\node[anchor=west, font=\footnotesize, black] at ({\R*cos(\angA)+0.1},{\R*sin(\angA)}) {$\alpha$};
\node[anchor=south east, font=\footnotesize, black] at ({\R*cos(\angB)-0.05},{\R*sin(\angB)+0.05}) {$\beta$};

\draw[thin, black] (0.45,0) arc[start angle=0, end angle=144, radius=0.45];
\node[font=\footnotesize, black] at (0.1,0.7) {$\tfrac{4\pi}{5}$};

\foreach \k in {0,...,9} {
  \pgfmathsetmacro{\ang}{\k*36}
  \pgfmathsetmacro{\rlab}{\R+0.22}
  \node[anchor=center, font=\scriptsize, black!70] at ({\rlab*cos(\ang)},{\rlab*sin(\ang)}) {$\zeta_{10}^{\k}$};
}

\end{tikzpicture}
\caption{\emph{The decagonal shell $S_{H_2}=\{\zeta_{10}^{k}:0\le k<10\}$ in the complex plane. The five dashed lines are the reflection axes of the dihedral Coxeter group of order ten. Two adjacent simple roots subtend the angle $4\pi/5$, giving the golden Cartan coefficient $2\langle\beta,\alpha\rangle/\langle\alpha,\alpha\rangle=2\cos(4\pi/5)=-\varphi$.}}
\label{fig:h2-decagon}
\end{figure}
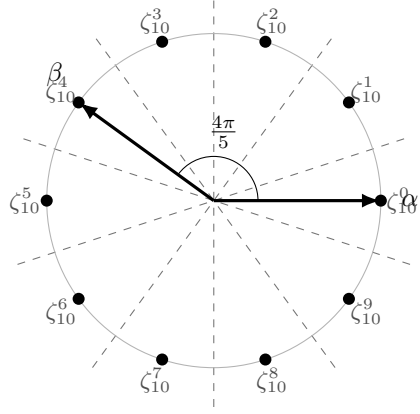

At this purely algebraic stage we use the name \emph{cyclotomic \(H_2\) integers}; the expression \emph{Penrose integers} is reserved for the situation in which the cut-and-project link with Penrose tilings has been specified.

\subsection{The icosian ring and \texorpdfstring{\(H_4\)}{H4}}

The quaternionic non-crystallographic example is the icosian ring. It is the golden analogue of the Hurwitz order, and its finite shell recovers \(H_4\) in the same way in which Coxeter--Dickson octonions recover \(E_8\).

Let \(\HH(K)\) be the quaternion algebra over \(K\) with basis \(1,i,j,k\). In the coordinate convention used in our computation, the \emph{icosian ring} \(\mathbb I\) is the free \(\Zphi\)-module of rank \(4\) generated by
\begin{equation}
e_1=1,\qquad e_2=i,\qquad e_3=\tfrac12(1+i+j+k),\qquad e_4=\tfrac12(-1+(\varphi-1)i-\varphi j).
\label{eq:icosian-basis}
\end{equation}
This is an \(H_4\)-compatible presentation of the standard icosian order \cite[Ch.~8 \S2]{Conway} used in the study of icosahedral quasicrystals \cite{Icosians} and the \(H_4\) root system \cite{Ti80}.

\begin{remark}
\label{rem:icosian-basis-generates-jk}
Each generator \(e_{r}\) has reduced quaternionic norm \(\Norm_{\HH}(e_{r})=1\), hence belongs to the \(120\)-element shell \(S_{H_{4}}\) of \(\mathbb I\). The basis \eqref{eq:icosian-basis} generates \(1,i,j,k\) over \(\Zphi\). Explicitly, using \(\varphi^{-1}=\varphi-1\in\Zphi\) and \((\varphi-1)^2=2-\varphi\),
\begin{align}
j&=(1-\varphi)e_1+(2-\varphi)e_2+2(1-\varphi)e_4, \notag\\
k&=2e_3+(\varphi-2)e_1+(\varphi-3)e_2+2(\varphi-1)e_4, \label{eq:icosian-jk-expressions}
\end{align}
so the \(\Zphi\)-span of \(\{e_1,\ldots,e_4\}\) contains \(1,i,j,k\) and hence agrees with the full integral icosian module of \cite[Ch.~8 \S2.1]{Conway}.
\end{remark}

\begin{proposition}
\label{prop:icosian-h4}
Fix the real embedding \(K\hookrightarrow\RR\) and the induced isometric inclusion \(\HH(K)\hookrightarrow\HH=\HH(K)\otimes_K\RR\). The norm-one shell of the icosian ring contains a distinguished \(120\)-element subset \(S_{H_4}\subset\mathbb I\) which, with the Euclidean form \(\langle x,y\rangle_{\RR}=\Rea(x\overline{y})\), is a realization of the root system \(H_4\).
\end{proposition}

\begin{proof}
Using \eqref{eq:icosian-basis} one verifies that the structure constants of \(\mathbb I\) under quaternion multiplication lie in \(\Zphi\), so \(\mathbb I\) is a \(\Zphi\)-order in \(\HH(K)\). After applying the real embedding, the \(120\) norm-one icosians are identified with the coordinate set \eqref{eq:h4-roots}, the standard realization of the \(600\)-cell vertices \cite{Conway,Ti80,Wi86}. The reflection group of \eqref{eq:h4-roots} is the Coxeter group of type \(H_4\), and the Cartan coefficient between adjacent simple roots is \(-\varphi\in\Zphi\). The same dihedral-orbit argument as in the proof of the previous proposition extends \eqref{eq:golden-cartan} to every pair of roots. The supplementary computation described in the Computational methods section verifies the structure constants and the reflection invariance of the \(120\) roots. \(\square\)
\end{proof}

In what follows, we will keep distinct the infinite order \(\mathbb I\), the finite \(H_4\) root shell, and the binary icosahedral group. They are related, but they play different roles.

\subsection{Icosian doubles as weak golden octonions}

A natural question is whether the octonionic analogue of the icosian construction is only a slogan or a genuine order. The cleanest positive answer is obtained by Cayley--Dickson doubling of the icosian ring.

\begin{definition}
\label{def:icosian-double}
Let \(\ell\) be a Cayley--Dickson generator satisfying \(\ell^2=-1\) and \(\ell a=\overline{a}\ell\) for \(a\in\HH(K)\); these are the conventions used throughout. Set
\begin{equation}
\mathbb G=\mathbb I\oplus \mathbb I\ell.
\label{eq:golden-octonion-double}
\end{equation}
The product is
\begin{equation}
(a+b\ell)(c+d\ell)=(ac-\overline d b)+(da+b\overline c)\ell,
\label{eq:cayley-dickson-product}
\end{equation}
with \(a,b,c,d\in\mathbb I\). The Cayley--Dickson conjugation is \(\overline{a+b\ell}=\overline{a}-b\ell\). We call \(\mathbb G\) the \emph{icosian double}.
\end{definition}

\begin{proposition}[Weak golden octonion order]
\label{prop:weak-golden-octonion}
The icosian double \(\mathbb G=\mathbb I\oplus\mathbb I\ell\) is a free \(\Zphi\)-order of rank \(8\) in \(\Oct(K)\). Moreover the finite shell
\begin{equation}
S_{\mathbb G}=(S_{H_4}\oplus0)\cup(0\oplus S_{H_4}\ell)
\label{eq:golden-octonion-shell}
\end{equation}
has \(240\) elements, is closed under negation and reflections, has Cartan coefficients in \(\Zphi\), and realizes the non-crystallographic root system \(H_4\oplus H_4\).
\end{proposition}

\begin{proof}
For \(a,b,c,d\in\mathbb I\) the components \(ac,\overline{d}b,da,b\overline{c}\) all lie in \(\mathbb I\) because \(\mathbb I\) is closed under quaternionic multiplication and conjugation. Hence \eqref{eq:cayley-dickson-product} maps \(\mathbb G\times\mathbb G\) into \(\mathbb G\), and the Cayley--Dickson conjugation \(\overline{a+b\ell}=\overline{a}-b\ell\) preserves \(\mathbb G\). The octonionic norm reduces to
\begin{equation}
\Norm(a+b\ell)=\Norm_{\HH}(a)+\Norm_{\HH}(b)\in\Zphi.
\label{eq:double-norm}
\end{equation}
Thus \(\mathbb G\) is a free \(\Zphi\)-order of rank \(8\).

The set \(S_{\mathbb G}\) splits as the orthogonal union of two copies of the \(H_4\) shell, one in the quaternionic part and one in the \(\ell\)-part. Its Gram matrix in the natural basis is block diagonal with two \(H_4\) Cartan matrices on the diagonal, so its Coxeter type is \(H_4\oplus H_4\). Reflection invariance and golden Cartan integrality follow blockwise from Proposition~\ref{prop:icosian-h4}. The supplementary computation verifies all \(240\) roots, all reflections, and the Cartan coefficients. Finally, the order is genuinely octonionic: for example,
\begin{equation}
[i,j,\ell]=(ij)\ell-i(j\ell)=2k\ell\neq0,
\label{eq:nonzero-associator}
\end{equation}
where \((ij)\ell=k\ell\) and \(i(j\ell)=(ji)\ell=-k\ell\) by \eqref{eq:cayley-dickson-product}. \(\square\)
\end{proof}

This is a positive but deliberately weak form of golden octonions. The shell \(H_4\oplus H_4\) is decomposable into two orthogonal icosian shells, even though the order generated by them is genuinely octonionic. A stronger golden Coxeter--Dickson twist, if it exists, would have to produce a less decomposable shell. The decomposition itself is sketched in Fig.~\ref{fig:h4-h4-decomposition}.

\begin{figure}[htbp]
\centering
\begin{tikzpicture}[
  every node/.style={font=\small},
  shell/.style={
    circle, very thick, draw=black,
    fill=white, minimum size=3.6cm, inner sep=0pt,
  },
  rootdot/.style={circle, fill=black, inner sep=0.9pt},
  flow/.style={-{Latex[length=2.5mm]}, very thick, black},
]
\node[shell] (L) at (-2.4,0) {};
\node[font=\footnotesize, anchor=center, black] at (-2.4,2.1) {$S_{H_4}\subset\mathbb I$};
\node[font=\footnotesize, anchor=center, black] at (-2.4,-2.1) {$120$ roots};

\foreach \k in {0,...,29} {
  \pgfmathsetmacro{\ang}{\k*12}
  \pgfmathsetmacro{\rr}{1.55}
  \node[rootdot] at ({-2.4 + \rr*cos(\ang)}, {\rr*sin(\ang)}) {};
}
\foreach \k in {0,...,19} {
  \pgfmathsetmacro{\ang}{\k*18+9}
  \pgfmathsetmacro{\rr}{1.05}
  \node[rootdot] at ({-2.4 + \rr*cos(\ang)}, {\rr*sin(\ang)}) {};
}
\foreach \k in {0,...,14} {
  \pgfmathsetmacro{\ang}{\k*24+5}
  \pgfmathsetmacro{\rr}{0.55}
  \node[rootdot] at ({-2.4 + \rr*cos(\ang)}, {\rr*sin(\ang)}) {};
}

\node[shell] (R) at (2.4,0) {};
\node[font=\footnotesize, anchor=center, black] at (2.4,2.1) {$S_{H_4}\ell\subset\mathbb I\ell$};
\node[font=\footnotesize, anchor=center, black] at (2.4,-2.1) {$120$ roots};

\foreach \k in {0,...,29} {
  \pgfmathsetmacro{\ang}{\k*12+6}
  \pgfmathsetmacro{\rr}{1.55}
  \node[rootdot] at ({2.4 + \rr*cos(\ang)}, {\rr*sin(\ang)}) {};
}
\foreach \k in {0,...,19} {
  \pgfmathsetmacro{\ang}{\k*18+3}
  \pgfmathsetmacro{\rr}{1.05}
  \node[rootdot] at ({2.4 + \rr*cos(\ang)}, {\rr*sin(\ang)}) {};
}
\foreach \k in {0,...,14} {
  \pgfmathsetmacro{\ang}{\k*24+15}
  \pgfmathsetmacro{\rr}{0.55}
  \node[rootdot] at ({2.4 + \rr*cos(\ang)}, {\rr*sin(\ang)}) {};
}

\draw[flow] (-0.6,0.15) -- (0.6,0.15) node[midway, above, font=\footnotesize, black] {$\cdot\,\ell$};
\draw[flow] (0.6,-0.15) -- (-0.6,-0.15) node[midway, below, font=\footnotesize, black] {$\cdot\,\overline{\ell}$};

\node[align=center, font=\footnotesize, black] at (0,-2.9)
  {$S_{\mathbb G}=(S_{H_4},\,0)\,\cup\,(0,\,S_{H_4}\ell)$\quad of type $H_4\oplus H_4$};

\end{tikzpicture}
\caption{\emph{The shell of the weak golden octonion order $\mathbb G=\mathbb I\oplus\mathbb I\ell$. It splits into two orthogonal copies of the icosian root shell $S_{H_4}$, exchanged by left multiplication by the Cayley--Dickson generator $\ell$. The resulting reflection geometry is the reducible Coxeter type $H_4\oplus H_4$. The order is genuinely octonionic, since the associator $[i,j,\ell]=2k\ell$ is nonzero, but the root shell itself is decomposable.}}
\label{fig:h4-h4-decomposition}
\end{figure}
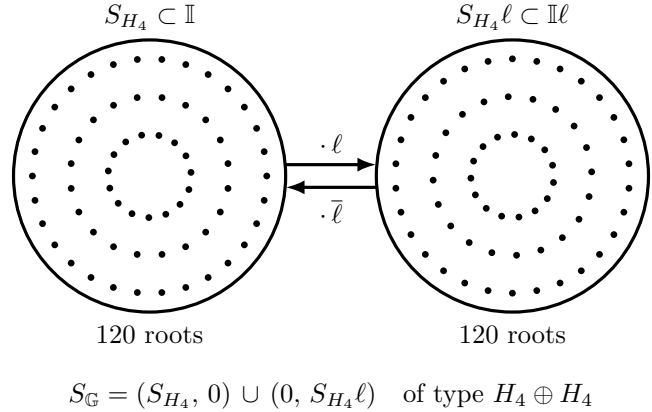

\subsection{The strong G3 problem}

The weak construction above does not exhaust the golden octonion problem. A stronger analogue of the Coxeter--Dickson order would have to contain a finite shell with genuinely mixed elements \(a+b\ell\), with \(a\neq0\) and \(b\neq0\), and it should not be equivalent to the orthogonal union of two independent \(H_4\) shells.

To state this precisely we use the following indecomposability convention.

\begin{definition}
\label{def:indecomposable}
A pair \((\mathcal O,S)\) consisting of a \(\Zphi\)-order \(\mathcal O\) and a finite shell \(S\subset\mathcal O\) is \emph{shell-decomposable} if there exists an orthogonal splitting \(A\otimes_K\RR=V_1\perp V_2\), both \(V_i\) nonzero and defined over \(K\) (i.e., obtained as \(V_i=W_i\otimes_K\RR\) for \(K\)-subspaces \(W_i\) of \(A\)), such that \(S=(S\cap V_1)\cup(S\cap V_2)\). Otherwise \((\mathcal O,S)\) is \emph{shell-indecomposable}. The weak icosian double \((\mathbb G,S_{\mathbb G})\) is shell-decomposable.
\end{definition}

We separate the strong G3 problem into increasingly rigid tests. A purely Coxeter-theoretic remark first explains why an irreducible \(H\)-type shell cannot live in real rank \(8\). Secondly, we test explicit fractional gluings of the weak double under the strict \(\Zphi\)-integral convention. Finally, we pass to the relaxed rational trace-integral convention.

\begin{remark}[Coxeter rank obstruction]
\label{rem:coxeter-rank}
The finite Coxeter classification \cite{Humphreys} produces only three irreducible non-crystallographic types: \(I_2(m)\), \(H_3\), and \(H_4\), of real rank \(2\), \(3\), \(4\) respectively. Hence no irreducible non-crystallographic finite Coxeter root system exists in rank \(8\); any shell-indecomposable golden octonion order of rank \(8\) must therefore have a finite shell which is not the root system of an irreducible Coxeter group.
\end{remark}

A useful counting criterion is the following.

\begin{lemma}[Mixed-projection criterion]
\label{lem:mixed-projection}
Let \(\RR^{8}=V_{1}\perp V_{2}\) with \(\dim V_{1}=\dim V_{2}=4\), and let \(S\subset\RR^{8}\) be a finite, centrally symmetric, reflection-invariant subset of a single norm value. Decompose each \(\alpha\in S\) as \(\alpha=\alpha^{(1)}+\alpha^{(2)}\) with \(\alpha^{(j)}\in V_{j}\), and call \(\alpha\) \emph{mixed} if both \(\alpha^{(1)}\) and \(\alpha^{(2)}\) are nonzero. Let \(M(S)\subset S\) be the set of mixed elements. Then \((\mathcal{O},S)\) is shell-indecomposable if and only if \(M(S)\neq\emptyset\); moreover in that case \(\#M(S)\geq 2|\mathrm{Stab}|\), where \(\mathrm{Stab}\) is the stabilizer of any non-mixed root.
\end{lemma}

\begin{proof}
The forward direction is the contrapositive of Definition~\ref{def:indecomposable}: if no element of \(S\) is mixed, then \(S\subset V_{1}\cup V_{2}\) and \(S=(S\cap V_{1})\cup(S\cap V_{2})\) is a shell decomposition, so \((\mathcal{O},S)\) is shell-decomposable. Conversely, if a single mixed root \(\alpha\) exists, then \(r_{\alpha}\) does not preserve \(V_{1}\) or \(V_{2}\); the reflection orbit of \(\alpha\) under the group generated by all \(r_{\beta}\) for \(\beta\in S\) is therefore not contained in either summand, and the decomposition fails. The lower bound \(\#M(S)\geq 2|\mathrm{Stab}|\) follows from central symmetry (\(\alpha\) and \(-\alpha\) both lie in \(M(S)\)) and from the action of the stabilizer of any non-mixed root \(\beta\), which preserves \(V_{1}\oplus V_{2}\) and therefore maps \(M(S)\) to itself with trivial fixed points. \(\square\)
\end{proof}

The lemma reduces the strong G3 problem, in its rank-\(8\) form, to producing at least one mixed-projection root, with the constraints of the chosen integrality convention. The remainder of this subsection shows that no such root can be added to \(\mathbb G\) under the natural integrality conventions.

\begin{proposition}[First bounded G3 no-go]
\label{prop:bounded-g3-no-go}
Let \(G_0=\mathbb I\oplus\mathbb I\ell\). Among all one-generator denominator-two gluings
\begin{equation}
G_v=G_0+\Zphi\frac{v}{2},
\qquad
\frac{v}{2}\in \frac12G_0/G_0,
\label{eq:single-line-gluing}
\end{equation}
no projective line \(v\in(\Zphi/2)^8\cong\mathbb F_4^8\) gives a strong G3 candidate satisfying the mixed-coset, conjugation, integral-pairing, integral-norm, multiplication, and square-closure tests.
\end{proposition}

\begin{proof}
The quotient \((\tfrac12 G_0)/G_0\) is the \(8\)-dimensional vector space over \(\Zphi/2\Zphi\cong\mathbb F_4\). Hence it has
\begin{equation}
\frac{4^8-1}{4-1}=21845
\label{eq:f4-lines}
\end{equation}
projective lines. We classify each line by the first necessary filter it fails:
\begin{enumerate}
\item[(F1)] \emph{not mixed}: the line lies entirely in one quaternionic half (\(v\in\mathbb I/2\mathbb I\) or \(v\in\mathbb I\ell/2\mathbb I\ell\)); count \(170\), corresponding to the projective lines of \((\mathbb F_4)^{4}\) in each half (\(85\) per half).
\item[(F2)] \emph{mixed but conjugation-stable failure}: the line is mixed but \(\overline{v/2}\notin G_0+\Zphi(v/2)\), so \(\Zphi v/2\) is not preserved by the Cayley--Dickson conjugation; count \(16320\).
\item[(F3)] \emph{mixed and conjugation-stable but integrally unpaired}: there exists a basis element \(g\in G_0\) with \(B(v/2,g)\notin\Zphi\); count \(5355\).
\end{enumerate}
The three counts sum to \(170+16320+5355=21845\), the total number of projective lines. Hence every projective line fails at least one of the three a priori conditions and the survivor count is zero. The supplementary script records, for one representative line in each class, the explicit obstruction (an element of \(G_0\) violating conjugation closure in class (F2), or a basis vector \(g\) for which \(B(v/2,g)\notin\Zphi\) in class (F3)). \(\square\)
\end{proof}

\begin{proposition}[Denominator-two subspace obstruction]
Let \(C\subset(\Zphi/2)^8\cong\mathbb F_4^8\) be an \(\mathbb F_4\)-linear subspace, and put
\begin{equation}
G_C=G_0+\frac12 C.
\label{eq:subspace-gluing}
\end{equation}
Under the strict \(\Zphi\)-order convention requiring integral pairings with the baseline basis of \(G_0\), no nonzero \(C\) gives a G3 candidate.
\end{proposition}

\begin{proof}
Every nonzero \(\mathbb F_4\)-subspace contains a projective line. If \(G_C\) satisfied the strict pairing condition, every element \(v/2\) with \(v\in C\) would pair integrally with the basis of \(G_0\). Hence the line \(\mathbb F_4v\) would be counted among the projective lines with integral pairings. The numerical verification records
\begin{equation}
21845\quad\hbox{projective lines tested},\qquad
0\quad\hbox{projective lines with integral pairings}.
\end{equation}
Therefore \(C\) must be zero by numerical verification of the projective-line obstruction.
\(\square\)
\end{proof}

\begin{proposition}[Dual-discriminant obstruction]
\label{prop:dual-discriminant}
Let
\begin{equation}
B(x,y)=\Norm(x+y)-\Norm(x)-\Norm(y)
\label{eq:norm-polar-pairing}
\end{equation}
be the polar bilinear form of the octonionic norm, and let
\begin{equation}
G_0^\#=\{x\in K\otimes_{\Zphi}G_0:\ B(x,G_0)\subset\Zphi\}.
\label{eq:g0-dual}
\end{equation}
For the icosian double \(G_0=\mathbb I\oplus\mathbb I\ell\), one has
\begin{equation}
G_0^\#=G_0.
\end{equation}
Consequently \(G_0\) has no strict norm-integral overorder obtained by adjoining a fractional ideal layer.
\end{proposition}

\begin{proof}
In the basis \((e_1,\ldots,e_4,e_1\ell,\ldots,e_4\ell)\) of \(G_0\) inherited from \eqref{eq:icosian-basis}, the Gram matrix of \(B\) is block diagonal with two identical \(4\times 4\) blocks \(G_{\mathbb I}\). A direct computation of \(B(e_i,e_j)=2\Rea(e_i\overline{e_j})\) gives
\begin{equation}
G_{\mathbb I}=
\begin{pmatrix}
2 & 0 & 1 & -1 \\
0 & 2 & 1 & \varphi-1 \\
1 & 1 & 2 & -1 \\
-1 & \varphi-1 & -1 & 2
\end{pmatrix},
\qquad
\det G_{\mathbb I}=\varphi^{2}=\varphi+1.
\label{eq:gram-icosian}
\end{equation}
Block-diagonality of the full pairing and \(\det G_{\mathbb I}=\varphi^{2}\) imply
\begin{equation}
G_{G_0}=G_{\mathbb I}\oplus G_{\mathbb I},
\qquad
\det G_{G_0}=(\det G_{\mathbb I})^{2}=\varphi^{4}=2+3\varphi.
\label{eq:gram-g0}
\end{equation}
The field norm of this determinant is
\begin{equation}
\Norm_{K/\QQ}(2+3\varphi)=(2+3\varphi)(5-3\varphi)=10+9\varphi-9\varphi^2=10+9\varphi-9(\varphi+1)=1,
\end{equation}
so \(\det G_{G_0}\) is a unit of \(\Zphi\). The adjugate \(\mathrm{adj}(G_{G_0})\) has entries in \(\Zphi\) because the cofactor of any minor of a matrix with entries in a commutative ring lies in that ring (see, e.g., Lang, \emph{Algebra}, Ch.~XIII~\S4 or the determinantal Cayley--Hamilton statement). Since \((\det G_{G_0})^{-1}\) lies in \(\Zphi\) by the unit computation above, the identity
\begin{equation}
G_{G_0}^{-1}=(\det G_{G_0})^{-1}\,\mathrm{adj}(G_{G_0})
\end{equation}
gives \(G_{G_0}^{-1}\in M_{8}(\Zphi)\). The pairing therefore identifies \(G_0\) with \(\operatorname{Hom}_{\Zphi}(G_0,\Zphi)\), and \(G_0^\#/G_0\) is trivial.

Now let \(G\supset G_0\) be an overorder of \(\Oct(K)\) with integral norm on every element. For \(x\in G\) and \(g\in G_0\),
\begin{equation}
B(x,g)=\Norm(x+g)-\Norm(x)-\Norm(g)\in\Zphi.
\end{equation}
Thus \(x\in G_0^\#=G_0\), and \(G=G_0\). \(\square\)
\end{proof}

The \(\ZZ\)-trace lattice gives a useful normalization check but not the proof of the preceding proposition. With respect to the trace-polar form
\begin{equation}
(x,y)_{\ZZ}=\Tr_{K/\QQ}\bigl(2\langle x,y\rangle_K\bigr),
\end{equation}
the \(\ZZ\)-lattice underlying \(\mathbb I\), as checked by numerical computation, is even and positive definite, but has determinant \(625\), not \(1\). Thus in the present normalization it is not literally the unimodular \(E_8\) lattice; the doubled trace lattice has determinant \(5^8\). The strict no-go theorem is therefore based on the stronger \(\Zphi\)-dual computation \(G_0^\#=G_0\).

\begin{proposition}[Ramified denominator check]
\label{prop:ramified-denominator}
Let \(\pi=\sqrt5=2\varphi-1\). Among all one-generator \(\pi\)-denominator gluings
\begin{equation}
G_v=G_0+\Zphi\frac{v}{\pi},
\qquad
v\in G_0/\pi G_0\simeq\mathbb F_5^8,
\end{equation}
no projective line satisfies the necessary polar-pairing condition. Hence no nonzero \(\mathbb F_5\)-subspace \(C\subset G_0/\pi G_0\) gives a strict norm-integral overmodule \(G_0+\pi^{-1}C\).
\end{proposition}

\begin{proof}
Since \(\Zphi/(\sqrt5)\cong\mathbb F_5\), there are
\begin{equation}
\frac{5^8-1}{5-1}=97656
\label{eq:f5-lines}
\end{equation}
projective lines in \(G_0/\pi G_0\). The polar pairing \(B\) modulo \(\sqrt 5\) is a non-degenerate quadratic form on \(\mathbb F_5^{8}\); the integral-pairing condition \(B(v,e_i)\in\pi\Zphi\) selects the lines lying in the radical of this reduced form, which is trivial. Hence the count of lines passing the polar-pairing filter is zero. The structural reason is therefore non-degeneracy of \(B\bmod\sqrt 5\), not exhaustive enumeration; the number \(97656\) is recorded only to confirm that the computation has visited the entire projective space. The number of mixed lines (\(97344\)) is reported for transparency but is not used: the polar-pairing filter is already decisive. Every nonzero \(\mathbb F_5\)-subspace contains a projective line, so the subspace statement follows. \(\square\)
\end{proof}

\begin{proposition}[Mixed half-root obstruction]
\label{prop:mixed-half-root}
For the direct ansatz
\begin{equation}
x=\frac{a+b\ell}{2},
\qquad a,b\in S_{H_4}\subset\mathbb I,
\label{eq:mixed-half-root}
\end{equation}
all \(14400\) raw pairs have norm \(1/2\). Hence none of them is an integral-norm element for the strict \(\Zphi\)-order convention.
\end{proposition}

\begin{proof}
Every element of the \(H_4\) shell has quaternionic norm \(1\). Therefore
\begin{equation}
\Norm\left(\frac{a+b\ell}{2}\right)=\frac{\Norm(a)+\Norm(b)}{4}
=\frac12.
\label{eq:mixed-half-root-norm}
\end{equation}
Since \(1/2\notin\Zphi\), such elements cannot be adjoined as strict integral-norm order elements. The numerical verification records \(14400\) pairs, \(3600\) projective denominator-two lines, and zero passing lines. \(\square\)
\end{proof}

The last obstruction is strict: it rejects elements whose \(K\)-valued norm is not in \(\Zphi\). Nevertheless, for quasicrystalline applications and for trace lattices it is natural to also consider rational trace integrality. This gives the relaxed G3-B' convention below.

\begin{definition}[Relaxed trace-integral G3-B']
\label{def:g3b}
A \emph{trace-integral golden octonion candidate} is a \(\Zphi\)-module \(G\subset\Oct(K)\), typically with \(G_0\subset G\), such that:
\begin{enumerate}
\item \(\Tr_{K/\QQ}\Norm(x)\in\ZZ\) for all \(x\in G\);
\item \(\Tr_{K/\QQ}B(x,y)\in\ZZ\) for all \(x,y\in G\), where \(B\) is the polar form \eqref{eq:norm-polar-pairing};
\item \(G\) is closed under octonion multiplication and conjugation;
\item \(G\) has a finite distinguished norm-trace shell;
\item that shell is reflection invariant and shell-indecomposable in the sense of Definition~\ref{def:indecomposable}.
\end{enumerate}
\end{definition}

\begin{remark}
\label{rem:g3b-nonvacuous}
The first four clauses of Definition~\ref{def:g3b} are non-vacuous: the icosian double \(G_0\) itself satisfies (1)--(4), with its \(240\)-element \(H_4\oplus H_4\) shell as the distinguished norm-trace shell. An obstruction to \(G_0\) being a G3-B' candidate is the indecomposability requirement (5); the propositions below show that the two natural fractional-gluing procedures cannot repair (5) while preserving (1)--(4).
\end{remark}

This relaxed problem is not covered by \(G_0^\#=G_0\), because it allows \(K\)-valued pairings and norms that are not elements of \(\Zphi\), provided their rational traces are integral. The mixed half-roots \((a+b\ell)/2\) are the first natural test objects: they fail the strict convention by having norm \(1/2\), but \(\Tr_{K/\QQ}(1/2)=1\).

\begin{lemma}[Trace-norm lattice]
Let
\begin{equation}
\Lambda=\{\alpha\in K:\Tr_{K/\QQ}(\alpha)\in\ZZ
\text{ and }\Tr_{K/\QQ}(\varphi^2\alpha)\in\ZZ\}.
\label{eq:trace-norm-lattice-definition}
\end{equation}
Then
\begin{equation}
\Lambda=\ZZ\varphi+\ZZ\frac{1}{\sqrt5}.
\label{eq:trace-norm-lattice}
\end{equation}
Consequently a \(\Zphi\)-line \(\Zphi v\), with \(\Norm(v)=\alpha\), has integral rational trace-norm on all its elements if and only if \(\alpha\in\Lambda\).
\end{lemma}

\begin{proof}
Write \(\alpha=a+b\varphi\). The two trace conditions are
\begin{equation}
m=\Tr_{K/\QQ}(\alpha)=2a+b,\qquad
n=\Tr_{K/\QQ}(\varphi^2\alpha)=3a+4b.
\label{eq:trace-norm-conditions}
\end{equation}
Solving gives
\begin{equation}
a=\frac{4m-n}{5},\qquad b=\frac{-3m+2n}{5}.
\label{eq:trace-norm-solution}
\end{equation}
Equivalently,
\begin{equation}
\alpha=m\varphi+\frac{n-4m}{\sqrt5},
\label{eq:trace-norm-generators}
\end{equation}
because \(1/\sqrt5=(-1+2\varphi)/5\). This proves the lattice formula.

For the final assertion, use the \(\ZZ\)-basis \(v,\varphi v\) of the line. The norm of \((r+s\varphi)v\) is \((r+s\varphi)^2\alpha\). The coefficients of the resulting binary trace form are \(\Tr(\alpha)\), \(2\Tr(\varphi\alpha)\), and \(\Tr(\varphi^2\alpha)\), and \(\Tr(\varphi\alpha)=\Tr(\varphi^2\alpha)-\Tr(\alpha)\). Hence the two defining conditions for \(\Lambda\) are necessary and sufficient. \(\square\)
\end{proof}

\begin{proposition}[\(H_4\) mixed half-root trace obstruction]
\label{prop:h4-mixed-trace}
The mixed half-root ansatz
\begin{equation}
v=\frac{a+b\ell}{2},\qquad a,b\in S_{H_4},
\end{equation}
does not produce a trace-integral \(\Zphi\)-module of the form \(G_0+\Zphi v\).
\end{proposition}

\begin{proof}
The numerical computation contains \(14400\) raw pairs and \(3600\) projective cosets. All raw pairs satisfy
\begin{equation}
\Tr_{K/\QQ}\Norm(v)=1.
\end{equation}
However, a \(\Zphi\)-module containing \(v\) also contains \(\varphi v\). For the representatives used in the numerical computation, the first structural failure is
\begin{equation}
\Norm(\varphi v)=\frac12+\frac12\varphi,\qquad
\Tr_{K/\QQ}\Norm(\varphi v)=\frac32.
\end{equation}
Equivalently, this is the obstruction \(1/2\notin\Lambda\) from the trace-norm lattice lemma, since \(\Tr_{K/\QQ}(\varphi^2/2)=3/2\). Consequently no projective coset generates a trace-integral \(\Zphi\)-module. The numerical verification records \(324\) raw candidates, or \(81\) projective cosets, passing the first polar-pairing filter with \(G_0\), but zero module survivors. \(\square\)
\end{proof}

\begin{proposition}[G3-B' discriminant-tower no-go]
\label{prop:g3-b-tower-no-go}
Inside the trace-polar \(\ZZ\)-dual quotient of \(G_0=\mathbb I+\mathbb I\ell\), there is no nonzero trace-integral octonion-stable gluing subspace. Equivalently, the first \(\sqrt5\)-denominator discriminant tower over the weak icosian double contains no G3-B' candidate.
\end{proposition}

\begin{proof}
The trace-polar discriminant quotient, obtained by numerical computation, is
\begin{equation}
G_{0,\ZZ}^{\#}/G_0\simeq(\ZZ/5\ZZ)^8.
\label{eq:g0-trace-discriminant}
\end{equation}
On this quotient the induced quadratic form is
\begin{equation}
q(t/5)=\frac{t^T G_{\ZZ}t}{50}\pmod{\ZZ}.
\label{eq:g0-discriminant-quadratic-form}
\end{equation}
The numerical verification identifies this form as split plus type \(O^+(8,5)\), of hyperbolic rank \(4\). It has \(97656\) projective lines and \(19656\) isotropic projective lines. Multiplication by \(\varphi\) acts as scalar \(3\), and multiplication by \(\sqrt5\) annihilates the quotient.

For each isotropic projective line, the numerical computation generates the smallest \(\mathbb F_5\)-subspace stable under conjugation and under left and right multiplication by the chosen \(\Zphi\)-basis of \(G_0\). In every one of the \(19656\) cases this stable closure has dimension \(8\), namely it is the whole discriminant quotient. The whole quotient is not totally isotropic. Therefore no nonzero isotropic subspace can also be stable under the octonionic operations required for an overmodule. The stable-subspace and candidate counts in the numerical verification are both zero.
Thus the G3-B' search inside the icosian-double discriminant tower has no candidate. This does not exclude golden octonion orders built from a different ambient order or from a different finite-shell convention. \(\square\)
\end{proof}

\section{Applications to quasicrystals, aperiodic algebras and physics}
\label{sec:applications}

The appearance of \(\Zphi\) is not accidental. Inflation by the golden ratio is one of the basic symmetries of Penrose tilings and icosahedral quasicrystals. The algebraic role of \(\Zphi\) is therefore the arithmetic counterpart of a geometric scaling symmetry.

It is of paramount importance to notice that the physical claim of the present work is modest. We do not claim that non-crystallographic integers directly produce a fundamental theory of matter. Rather, they provide exact arithmetic data which can be transported into quasicrystalline and aperiodic models. The useful chain is
\begin{equation}
\begin{gathered}
\text{root shell}\longrightarrow\text{cut-and-project module}\\
\longrightarrow\text{reciprocal module}\longrightarrow\text{computable observable}.
\end{gathered}
\label{eq:physics-pipeline}
\end{equation}
In this sense the finite shell is the algebraic input, while diffraction peaks, spectra, density of states, inverse participation ratios and phason modes are the physical outputs. The pipeline is sketched in Fig.~\ref{fig:physics-pipeline}.

\begin{figure}[htbp]
\centering
\begin{tikzpicture}[
  every node/.style={font=\small},
  box/.style={
    rectangle, rounded corners=4pt, draw=black, very thick,
    align=center, minimum height=1.4cm, minimum width=2.5cm,
    inner sep=4pt, fill=white,
  },
  flow/.style={-{Latex[length=3mm]}, very thick, black},
]
\node[box] (a) at (0,0)
  {root shell\\$H_2,\,H_3,\,H_4,\,\mathbb I,\,\mathbb G$\\over $\Zphi$};
\node[box] (b) at (4.0,0)
  {cut-and-project\\module $M\subset\RR^{d}$};
\node[box] (c) at (8.0,0)
  {reciprocal\\module $M^{*}$};
\node[box] (d) at (12.2,0)
  {observable\\$I(k),\,\mathrm{DOS},\,\mathrm{IPR},$\\phason modes};

\draw[flow] (a) -- (b);
\draw[flow] (b) -- (c);
\draw[flow] (c) -- (d);

\node[font=\footnotesize\itshape, black] at (0,-1.3) {arithmetic input};
\node[font=\footnotesize\itshape, black] at (4.0,-1.3) {physical patch};
\node[font=\footnotesize\itshape, black] at (8.0,-1.3) {Bragg labels};
\node[font=\footnotesize\itshape, black] at (12.2,-1.3) {computable physics};
\end{tikzpicture}
\caption{\emph{The operational pipeline from non-crystallographic arithmetic to physical observables. A finite root shell over $\Zphi$ is promoted to a cut-and-project module in physical space; its reciprocal module supplies the labels of diffraction peaks; and tight-binding spectra, densities of states, inverse participation ratios and phason modes are computed on the resulting aperiodic point set.}}
\label{fig:physics-pipeline}
\end{figure}
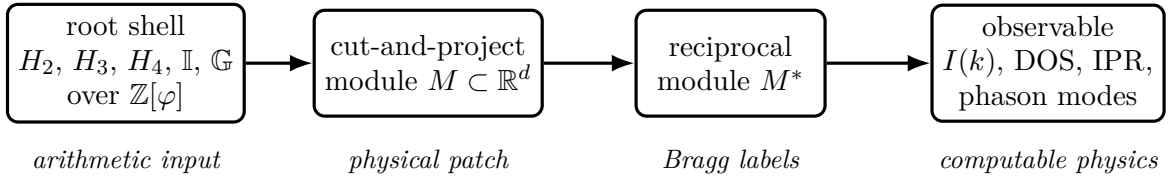

\subsection{Quasicrystal modules and diffraction}

For \(H_2\), the cyclotomic shell \(\ZZ[\zeta_{10}]\) and its real subring \(\Zphi\) give the rank-two algebraic structure behind decagonal symmetry. After choosing physical and internal spaces, this is the standard cut-and-project setting for Penrose-type tilings \cite{DeBr81,Pen89,Aperiodic}. In particular, the multiplication by \(\varphi\) gives the inflation symmetry, whereas the \(\ZZ\)-module generated by the shell gives the possible wave-vectors of the reciprocal module.

If \(\Lambda\) is a finite patch obtained from such a cut-and-project set, a first observable is the numerical diffraction intensity
\begin{equation}
I(k)=\left|\sum_{x\in\Lambda}w_x e^{-2\pi i k\cdot x}\right|^2,
\label{eq:diffraction-intensity}
\end{equation}
where \(w_x\) is the scattering weight at the point \(x\). Formula \eqref{eq:diffraction-intensity} is not new by itself. What the present formalism adds is an exact arithmetic control of the allowed \(k\)-module and of its scaling under \(\varphi\). The expected pattern is sketched in Fig.~\ref{fig:decagonal-diffraction}.

\begin{figure}[htbp]
\centering
\begin{tikzpicture}[
  every node/.style={font=\small},
  scale=1.6,
  peakbig/.style={circle, fill=black, inner sep=1.2pt},
  peakmed/.style={circle, fill=black, inner sep=0.75pt},
  peaksml/.style={circle, fill=black, inner sep=0.45pt},
  scale10/.style={->, black!60, thick, font=\footnotesize},
]
\pgfmathsetmacro{\PH}{(1+sqrt(5))/2}
\pgfmathsetmacro{\IPH}{1/\PH}

\draw[black!25, very thin] (-2.8,-2.8) rectangle (2.8,2.8);

\foreach \k in {0,...,9} {
  \pgfmathsetmacro{\ang}{\k*36}
  \node[peakmed] at ({\PH*cos(\ang)},{\PH*sin(\ang)}) {};
}
\foreach \k in {0,...,9} {
  \pgfmathsetmacro{\ang}{\k*36 + 18}
  \node[peakbig] at ({1.0*cos(\ang)},{1.0*sin(\ang)}) {};
}
\foreach \k in {0,...,9} {
  \pgfmathsetmacro{\ang}{\k*36}
  \node[peakmed] at ({\IPH*cos(\ang)},{\IPH*sin(\ang)}) {};
}
\foreach \k in {0,...,9} {
  \pgfmathsetmacro{\ang}{\k*36 + 18}
  \pgfmathsetmacro{\rr}{\PH*\PH}
  \node[peaksml] at ({\rr*cos(\ang)},{\rr*sin(\ang)}) {};
}
\foreach \k in {0,...,9} {
  \pgfmathsetmacro{\ang}{\k*36 + 9}
  \node[peaksml] at ({0.66*cos(\ang)},{0.66*sin(\ang)}) {};
  \node[peaksml] at ({1.3*cos(\ang)},{1.3*sin(\ang)}) {};
  \node[peaksml] at ({2.1*cos(\ang)},{2.1*sin(\ang)}) {};
}
\node[peakbig] at (0,0) {};

\draw[scale10] (0,0) -- ({\PH*cos(60)},{\PH*sin(60)});
\node[font=\footnotesize, anchor=west, black] at ({\PH*cos(60)},{\PH*sin(60)+0.2}) {$|k|=\varphi$};
\draw[scale10] (0,0) -- ({(\PH*\PH)*cos(120)},{(\PH*\PH)*sin(120)});
\node[font=\footnotesize, anchor=east, black] at ({(\PH*\PH)*cos(120)},{(\PH*\PH)*sin(120)+0.2}) {$|k|=\varphi^{2}$};

\draw[->, black!35, thin] (-2.8,0) -- (2.95,0) node[right, font=\footnotesize, black]{$k_x$};
\draw[->, black!35, thin] (0,-2.8) -- (0,2.95) node[above, font=\footnotesize, black]{$k_y$};
\end{tikzpicture}
\caption{\emph{Stylized decagonal Bragg diffraction pattern produced by a $\Zphi$-indexed quasicrystal point set. Peaks are located on concentric ten-fold orbits with radii in the multiplicative chain $\ldots,\varphi^{-1},1,\varphi,\varphi^{2},\ldots$; peak intensities decay with radius and reflect the $\varphi$-inflation symmetry of the underlying cut-and-project module. The ten-fold rotational symmetry forbids any periodic lattice indexation but is naturally captured by the reciprocal $\Zphi$-module.}}
\label{fig:decagonal-diffraction}
\end{figure}
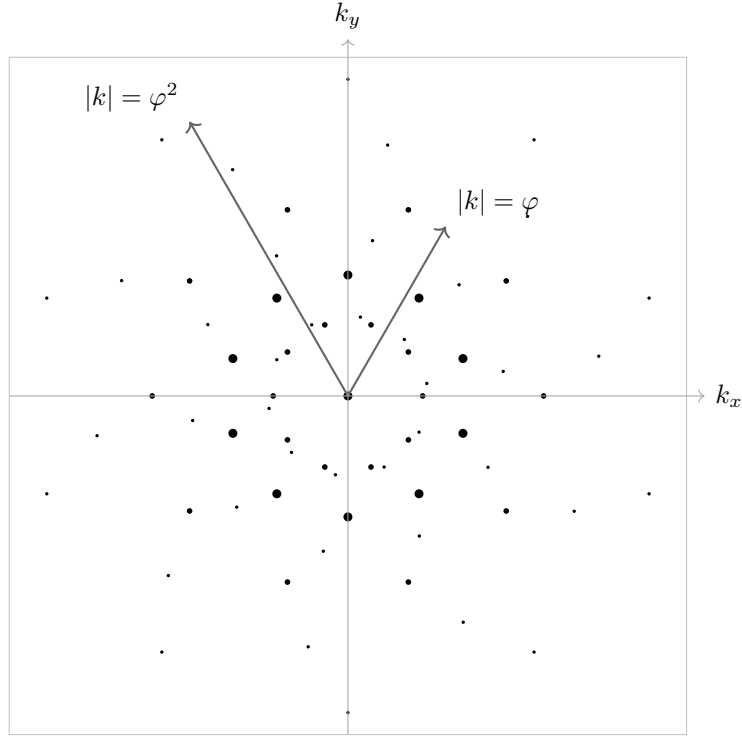

The three-dimensional system \(H_3\) has a more direct physical interpretation. Its \(30\) roots give the icosidodecahedral star vectors used to model icosahedral orientation data and diffraction directions. The four-dimensional system \(H_4\), on the other hand, should be viewed as the arithmetic parent. The icosian shell gives the \(600\)-cell, and projections from it recover three-dimensional icosahedral clusters and reciprocal stars. Moody and Patera showed that icosians provide a natural arithmetic framework for quasicrystals \cite{Icosians}. In the present language, this means that the icosian order stores the finite shell, while the shell stores the non-crystallographic reflection geometry.

\begin{remark}
The usual slogan that ``icosians give \(E_8\)'' hides a normalization issue which is relevant for applications. In the trace form used in our numerical computation, the icosian lattice is \(5\)-modular rather than unimodular. Therefore a physical model which is sensitive to discriminant data should not identify it blindly with an \(E_8\) lattice. The root data may agree, while the arithmetic layer differs.
\end{remark}

\subsection{Spectral and elastic models}

Once a finite quasicrystal patch has been generated from a non-crystallographic integer system, one can pass from geometry to dynamics. The most immediate model is a tight-binding Hamiltonian on the adjacency graph of the patch:
\begin{equation}
H=-t\sum_{\langle x,y\rangle}|x\rangle\langle y|
\;+\;\sum_x V_x |x\rangle\langle x|.
\label{eq:tight-binding}
\end{equation}
Here the edge relation \(\langle x,y\rangle\) comes from the tiling or from a chosen distance shell, and \(V_x\) may depend on the local aperiodic environment. The quantities to compute are then the spectrum, the density of states and the inverse participation ratio
\begin{equation}
\operatorname{IPR}(\psi)=\sum_x|\psi(x)|^4.
\label{eq:ipr}
\end{equation}
The role of \(\Zphi\) is to give exact labels for inflation, local environments and candidate gap labels.

Similarly, elastic or phononic models can be built from the same point sets. A discrete harmonic-spring energy has the form
\begin{equation}
E(u)=\sum_{\langle x,y\rangle}k_{xy}
\left(|u_x-u_y|-\ell_{xy}\right)^{2},
\label{eq:elastic-energy}
\end{equation}
or, in its linearized version about the rest configuration,
\begin{equation}
E_{\mathrm{lin}}(u)=\tfrac12\sum_{\langle x,y\rangle}k_{xy}\bigl(\hat n_{xy}\cdot(u_x-u_y)\bigr)^{2},
\end{equation}
where \(\hat n_{xy}\) is the unit vector along the equilibrium bond.
In a periodic crystal, the normal modes are organized by a reciprocal lattice. In a quasicrystal, the reciprocal object is a module, and the internal cut-and-project space gives phason degrees of freedom. The arithmetic developed above supplies exact coordinates for both the physical and internal components.

\subsection{Aperiodic algebras}

There is a second, more algebraic, direction. Aperiodic Lie and Jordan algebras built over quasicrystalline point sets can be viewed as later layers over the same arithmetic \cite{CCAI23,CCAI23b,CCAI23c}. The present framework clarifies what the input should be: not a vague non-periodic set, but a finite non-crystallographic shell together with its \(\Zphi\)-Cartan coefficients and its cut-and-project module.

In analogy to the octonionic case, the icosian double gives a natural next test object. The order \(G_0=\mathbb I+\mathbb I\ell\) is genuinely octonionic, since its associator is nonzero, but its visible shell is still \(H_4\oplus H_4\). Thus the physically honest interpretation is not a new particle model, but a non-associative toy model of two coupled icosian sectors. The basic observable is the associator
\begin{equation}
A(x,y,z)=(xy)z-x(yz),
\label{eq:associator-observable}
\end{equation}
measured on triples of shell elements. Comparing the distribution of \(\Norm(A(x,y,z))\) with the quaternionic \(H_4\) case would isolate precisely what is new in the octonionic double.

\begin{remark}
This distinction is important. The strict and relaxed G3 searches above are negative inside the present icosian-double tower. Therefore the weak golden octonion order should be used in physics as a controlled non-associative model, not as evidence for an indecomposable golden Coxeter--Dickson shell.
\end{remark}

\subsection{Practical hierarchy of applications}

The roadmap emerging from the previous discussion is the following.

\begin{table}[htbp]
\centering
\resizebox{\textwidth}{!}{%
\begin{tabular}{llll}
\toprule
Direction & Arithmetic input & Observable & Status \\
\midrule
Penrose/decagonal models & \(H_2\), \(\ZZ[\zeta_{10}]\), \(\Zphi\) & diffraction, inflation, tight-binding spectrum & safest first target \\
Icosahedral clusters & \(H_3\) shell & star vectors, orientation states, 3D diffraction & direct physical geometry \\
Icosian quasicrystals & \(H_4\), \(\mathbb I\), trace lattice & projections, reciprocal modules, \(K_8/E_8\) comparison & arithmetic refinement \\
Aperiodic algebras & \(H_2,H_3,H_4\) modules & brackets, Jordan products, derivations & CCAI continuation \\
Golden octonion toy models & \(G_0= \mathbb I+\mathbb I\ell\), \(H_4\oplus H_4\) & associator distribution, coupled sectors & speculative but testable \\
\bottomrule
\end{tabular}
}
\caption{\emph{This table summarizes the hierarchy of physical applications suggested by the non-crystallographic integer framework.}}
\label{tab:physics-applications}
\end{table}

The first publishable physical application should therefore be the \(H_2\) Penrose case: it is the most concrete, it has a standard cut-and-project interpretation, and it immediately gives numerical observables such as \eqref{eq:diffraction-intensity}, \eqref{eq:tight-binding} and \eqref{eq:ipr}. The \(H_3\) and \(H_4\) layers then extend the same arithmetic to icosahedral geometry. Only after these cases have been computed should the weak golden octonion order be used for physical toy models.

\section{Conclusions and Future Developments}

In this work we defined a root-shell formulation of non-crystallographic systems of integers over composition algebras. The construction recovers the classical crystallographic examples, explains the role of finite norm shells, and shows why \(\Zphi\) is the natural coefficient ring for the \(H_2\), icosian \(H_4\), and icosian-double \(H_4\oplus H_4\) examples. In particular, the icosian double gives a genuine octonionic \(\Zphi\)-order, albeit with a decomposable \(H_4\oplus H_4\) shell.

The analogy with the octonionic case is preserved once the lattice requirement is replaced by the root-shell requirement. Coxeter--Dickson octonions recover the crystallographic shell \(E_8\); icosians recover the non-crystallographic shell \(H_4\). The cost of the analogy is the requirement to distinguish four objects which the crystallographic case allows one to conflate: the full order, the finite shell, the multiplicative unit object and the arithmetic coefficient ring.

From the physical point of view, the most robust application is the quasicrystalline one. The \(H_2\) shell gives exact Penrose arithmetic, \(H_3\) gives icosahedral star vectors, and \(H_4\) gives the icosian arithmetic parent. This produces concrete observables such as diffraction intensities, reciprocal-module labels, tight-binding spectra and phason modes. The weak golden octonion order enters only later, as a controlled non-associative model of coupled icosian sectors through the associator.

A natural next direction is to ask whether a stronger golden Coxeter--Dickson analogue can be made precise outside the present icosian-double tower. The strict overorder problem over \(G_0\) is closed by \(G_0^\#=G_0\), and the first trace-integral discriminant tower has no octonion-stable isotropic gluing. The remaining open directions are therefore: different ambient golden octonion orders, twisted Cayley--Dickson doubles, neighbor constructions in the sense of Kneser, and a first physical implementation of the \(H_2\) Penrose pipeline. The present article fixes the arithmetic language before further work can address the physical and octonionic ramifications.

\section*{Acknowledgments}

The author thanks Raymond Aschheim for discussions on integral numbers, and Richard Clawson,
David Chester, and Klee Irwin for discussions on lattices and root systems.

\section*{Computational methods}

All numerical statements quoted in the text rest on the same exact-arithmetic computational pipeline, implemented in Python~3.11 with no floating-point dependence. Elements of \(\Zphi\) are represented as pairs of integers \((a,b)\) standing for \(a+b\varphi\), with multiplication reduced by \(\varphi^2=\varphi+1\). Elements of \(K=\QQ(\sqrt5)\) are represented as pairs of \texttt{fractions.Fraction} values. The icosian ring \(\mathbb I\) is represented by the basis \eqref{eq:icosian-basis}; quaternionic multiplication and conjugation are encoded by their \(4\times 4\) structure constants in \(\Zphi\). The icosian double \(G_0=\mathbb I\oplus\mathbb I\ell\) is represented by the \(8\)-tuple of \(\Zphi\)-coefficients, with multiplication implemented via the Cayley--Dickson formula \eqref{eq:cayley-dickson-product}.

The verification protocol consists of the following exact checks.

\begin{enumerate}
\item[(P1)] Closure of \(\mathbb I\) and \(G_0\) under multiplication and conjugation; integrality of \(\Tr_A\) and \(\Norm_A\) on every basis pair.
\item[(P2)] Enumeration of the \(120\)-element \(H_4\) shell of \(\mathbb I\) by exhaustive search over the bounded coordinate box dictated by \eqref{eq:h4-roots}, followed by verification of \(S_{H_4}=-S_{H_4}\) and of the reflection action.
\item[(P3)] Computation of the Gram matrix \(G_{\mathbb I}\) of the polar pairing, including bit-by-bit comparison with the matrix displayed in \eqref{eq:gram-icosian}, computation of its determinant \(\det G_{\mathbb I}=\varphi^{2}\), and of the inverse matrix \(G_{\mathbb I}^{-1}\); confirmation that \(G_{\mathbb I}^{-1}\) has entries in \(\Zphi\) by direct cofactor expansion.
\item[(P4)] For the strict denominator-two G3 search: enumeration of all \(21845\) projective lines of \((\Zphi/2\Zphi)^{8}\simeq\mathbb F_{4}^{8}\), classification by the (mixed, conjugation-closed, integrally paired) filters, and confirmation of zero survivors.
\item[(P5)] For the ramified \(\sqrt5\) search: enumeration of all \(97656\) projective lines of \((\Zphi/\sqrt5\Zphi)^{8}\simeq\mathbb F_{5}^{8}\) and confirmation of zero survivors of the polar-pairing filter.
\item[(P6)] For the trace-polar discriminant tower: computation of \(G_{0,\ZZ}^{\#}/G_0\simeq(\ZZ/5\ZZ)^{8}\), classification of the induced quadratic form as split type \(O^{+}(8,5)\) (hyperbolic rank \(4\)), enumeration of the \(19656\) isotropic projective lines, and verification that each generates a closure of full dimension under conjugation and multiplication by the \(\Zphi\)-basis.
\end{enumerate}

The implementation uses only the Python standard library (in particular \texttt{fractions.Fraction} and the built-in \texttt{int}); no third-party dependencies are required. The complete source code, the integer certificates produced by each check, and a SHA-256 hash of the certificate archive will be deposited as a supplementary file accompanying this article; running the verification end-to-end takes under one minute on a standard laptop. Independent reproduction of all numerical statements amounts to executing the supplied verification script and comparing the certificate hash.

\end{document}